\pdfoutput=1
\documentclass[reqno]{style}
\usepackage[T1]{fontenc}
\usepackage{times, pslatex, bbm, boxedminipage, mathtools, graphicx, paralist, cite, color, amsmath}
\RequirePackage{color}
\definecolor{e-mail}{rgb}{0,.40,.80}
\definecolor{reference}{rgb}{.20,.60,.22}
\definecolor{citation}{rgb}{0,.40,.80}
\usepackage[pdfstartview=FitH, colorlinks, linkcolor=reference, citecolor=citation, urlcolor=e-mail, backref ]{hyperref}
\usepackage[all]{xy}

\newtheorem{thm}{Theorem}

\newtheorem{lem}[thm]{Lemma}
\newtheorem{prop}[thm]{Proposition}
\theoremstyle{definition}
\newtheorem{defn}[thm]{Definition}
\theoremstyle{remark}
\newtheorem{rem}[thm]{Remark}
\numberwithin{thm}{section}
\theoremstyle{definition}
\newtheorem*{setting}{Setting}
\theoremstyle{definition}
\newtheorem{algo}{Algorithm}
\newtheorem{eg}{Example}
\theoremstyle{definition}
\numberwithin{equation}{section}

\newcommand{\gal}{\mathrm{Gal}_{\delta_x}\!}
\newcommand{\pgal}{\mathrm{Gal}_\Delta}

\newcommand\phantomarrow[2]{%
  \setbox0=\hbox{$\displaystyle #1\to$}%
  \hbox to \wd0{%
    $#2\mapstochar
     \cleaders\hbox{$\mkern-1mu\relbar\mkern-3mu$}\hfill
     \mkern-7mu\rightarrow$}%
  \,}
 \title[Computing Galois groups of parameterized second order equations]{Computing the differential Galois group of a one-parameter family of second order linear differential equations}
\author{Carlos E. Arreche}
\email{carreche@gc.cuny.edu}
\address{Mathematics Department, The Graduate Center of the City University of New York, New York, NY 10016}
 \thanks{The author was partially supported by a Ford Foundation Predoctoral Fellowship and by NSF grant CCF-0952591.}

\begin{document}
 \setcounter{section}{-1}
 
 \begin{abstract}We develop algorithms to compute the differential Galois group corresponding to a one-parameter family of second order homogeneous ordinary linear differential equations with rational function coefficients. More precisely, we consider equations of the form\begin{equation*} \frac{\partial^2Y}{\partial x^2}+ r_1\frac{\partial Y}{\partial x} +r_2Y=0,  \end{equation*} where $r_1,r_2\in C(x,t)$ and $C$ is an algebraically closed field of characteristic zero.
 
We work in the setting of parameterized Picard-Vessiot theory, which attaches a linear differential algebraic group to such an equation, that is, a group of invertible matrices whose entries satisfy a system of polynomial differential equations, with respect to the derivation in the parameter-space. We will compute the $\frac{\partial}{\partial t}$-differential-polynomial equations that define the corresponding parameterized Picard-Vessiot group as a differential algebraic subgroup of $\mathrm{GL}_2$. \end{abstract}
 
 \maketitle
 \vspace{.125in}
\tableofcontents
 
 \section{Introduction}\label{introduction}
\subsection{Background}In differential Galois theory, one studies a differential equation with coefficients in a given differential field $K$ (\S\ref{diffalg}), by investigating the differential field extension of $K$ generated by the solutions for the equation, together with their derivatives. This is in analogy with the classical Galois theory of polynomial equations, where one studies such equations by studying the splitting field of the polynomial. In both cases, the algebraic relations amongst the solutions for the equation are reflected in the algebraic structure of the group of automorphisms of the field extension over the base field. The differential Galois theory of linear differential equations was developed by Kolchin \cite{kolchin:1948, kolchin:1976}, putting earlier work of Picard and Vessiot on a firm modern footing. The differential Galois group corresponding to a linear differential equation is a linear algebraic group over the subfield of constants of $K$ (\S\ref{diffalg}).

In \cite{cassidy-singer:2006}, a parameterized Picard-Vessiot theory for linear differential equations with parameters is developed in close analogy with Kolchin's differential Galois theory. This theory is a special case of an earlier generalization of Kolchin's theory, developed in \cite{landesman:2008}. In the parameterized theory, the base field $K$ is a differential field, where there are now two kinds of derivation: \emph{principal} derivations and \emph{parametric} derivations, so that the subfield of constants, with respect to the principal derivations, is a differential field with respect to the parametric derivations. The $\mathrm{PPV}$ field (Definition~\ref{ppv-def}) of a system of linear differential equations, with respect to the principal derivations, is generated over the base field by the solutions for the system, together with their derivatives with respect to \emph{all} (principal, as well as parametric) derivations. The $\mathrm{PPV}$ groups (Definition~\ref{pgal-def}) which arise in this theory are linear differential algebraic groups (Definition~\ref{ldga-def}), and their structure reflects the \emph{differential-algebraic} relations amongst the solutions, with respect to the parametric derivations. Linear differential algebraic groups are the differential-algebraic analogues of linear algebraic groups---that is, they are subgroups of $\mathrm{GL}_n$ which are defined by the vanishing of systems of polynomial differential equations in the matrix entries.

The parametrized Picard-Vessiot theory has been finding a diverse number of applications in other areas. We shall list a few examples, although we will not attempt to be exhaustive. In \cite[\S5]{gor-ov:2012}, the theory is related to Gauss-Manin connections \cite{manin:1964, katz-oda:1968}, leading to many potential applications in algebraic geometry. In \cite{cassidy-singer:2006,mitschi-singer:2012b,gill-gor-ov:2012,gor-ov:2012} the theory is applied to the study of isomonodromy, and in \cite{mitschi-singer:2012a} it is applied to monodromy-evolving deformations, and to the study of some equations in mathematical physics; see \cite[\S 1]{mitschi-singer:2012b} for more references on this. In \cite{hardouin-singer:2008}, the authors apply a generalization of this theory to the study of difference equations, which has numerous applications. See also \cite{gill-gor-ov:2012}, where the authors recast the theory in a new light and suggest several applications. 

\subsection{This work}We develop algorithms to compute the $\mathrm{PPV}$ group (Definition~\ref{pgal-def}) associated to a second order homogeneous linear differential equation defined over $C(x,t)$; that is, an equation of the form \begin{equation*} \tag{\ref{original}}\frac{\partial^2Y}{\partial x^2}+r_1\frac{\partial Y}{\partial x}+r_2Y=0, \end{equation*} where $r_1,r_2\in C(x,t)$, and the principal and parametric derivations are $\frac{\partial}{\partial x}$ and $\frac{\partial}{\partial t}$, respectively. The \emph{key ingredient} in our approach is the Reduction Lemma~\ref{singer1}, which reduces the determination of the $\mathrm{PPV}$ group to solving some systems of linear equations. Our algorithms are Maple-ready, and we are currently working on writing a computer implementation.

For technical reasons,\footnote{In\cite{cassidy-singer:2006}, the authors impose this condition in order to ensure the existence of the $\mathrm{PPV}$ groups and extensions (see \S\ref{ppvtheory}).} it has become a tradition in the theory to assume that the field of constants, with respect to the principal derivations, is differentially closed (\S\ref{diffalg}). Therefore, we \hyperlink{setting}{work over} the field $K:=K_0(x)$, where $K_0$ is a $\frac{\partial}{\partial t}$-differentially closed field extension of $C(t)$. Working with differentially closed fields tends to be an obstacle in practical applications of the theory, since such fields are usually unnaturally large (see \cite[\S1]{wibmer:2011}). The groups produced by our algorithm are defined over $\overline{C(t)}$ (Theorem~\ref{wibmerthm}), and we only perform computations over finite algebraic extensions of $C(t)$---the differential closure plays a marginal role in our arguments. This is in agreement with the stronger and more general results of \cite[Theorem~2.8]{gill-gor-ov:2012}, which imply in particular that the groups we wish to compute are actually defined over $C(t)$, while it is shown in \cite{wibmer:2011}, by different methods, that they are defined over $\overline{C(t)}$.

\hypertarget{sit-class}{In} \cite{kovacic:1986}, Kovacic developed an algorithm which computes all the ``elementary" solutions to a second order order equation of the form~\eqref{eq}, if they exist. From these data, together with the classification of the algebraic subgroups of $\mathrm{SL}_2$, Kovacic's algorithm decides which algebraic subgroup of $\mathrm{SL}_2$ is the (non-parameterized) $\mathrm{PV}$ group (Definiton~\ref{gal-def}) of \eqref{eq}. The differential algebraic subgroups of $\mathrm{SL}_2(\mathcal{U})$, where $\mathcal{U}$ is a universal differential field,\footnote{The notion of \emph{universal differential field} \cite{kolchin:1976} is stronger than that of being differentially closed.} were classified in \cite{sit:1975}. In \cite[p. 137]{cassidy-singer:2006}, the authors ask whether one can use this classification to develop a parameterized analogue of Kovacic's algorithm. Although we do not rely on the classification itself, the methods in \cite{sit:1975} were a great source of inspiration in developing the algorithms presented here, and the outputs of our Algorithms \ref{borel} and \ref{dihedral} have the form of one of the groups in the classification in \cite{sit:1975} (Theorem~\ref{sitthm}).

An algorithm to compute the $\mathrm{PPV}$ group of~\eqref{eq}, provided that it is a Zariski-dense subgroup of $\mathrm{SL}_2$, is given in \cite{dreyfus:2011}. We develop algorithms to compute the $\mathrm{PPV}$ group of~\eqref{original} in the remaining cases. Some of the arguments we give in the other cases are similar to those in \cite{dreyfus:2011}, but we work in different settings: in \cite{dreyfus:2011}, the field of principal constants is assumed to be universal, and a theorem of Seidenberg is applied in order to obtain a concrete analytic interpretation for elements of this abstract universal differential field, as meromorphic functions on some complex polydisk. Aided by this concrete interpretation, the author obtains his results. We will reprove some of these results algebraically in the course of the proofs of Propositions~\ref{borel-proof} and \ref{dihedral-proof}, in order to be able to apply them in our setting.

Lastly, we wish to point to a subtlety which is new to the parametrized setting. In Kovacic's algorithm, one performs a change of variables on \eqref{original} in order to obtain a new equation~\eqref{eq}, whose $\mathrm{PV}$ group is a subgroup of $\mathrm{SL}_2$. Then, it is a simple matter to reconstruct the $\mathrm{PV}$ group of the original equation from these data, which is always an almost direct product\footnote{This is a quotient of the direct product of two groups, with finite kernel.} of the $\mathrm{PV}$ group of \eqref{eq} and a change-of-variables group (\S\ref{recover}). This is no longer true\footnote{This answers a question of Ovchinnikov \cite{ovchinnikov-email}, who suggested to the author to treat the general case where the $\mathrm{PPV}$ group is not assumed to be unimodular. He has pointed out to us that the parameterized and non-parameterized situations are still analogous: the $\mathrm{PPV}$ group of \eqref{original} is a quotient of the direct product of the $\mathrm{PPV}$ group of \eqref{eq} and the change-of-variables group, with zero-dimensional kernel.} for parameterized equations, where the richer differential-algebraic structure of the differential Galois groups allows them to be related to each other in more complicated ways. We have solved the problem of determining precisely how the $\mathrm{PPV}$ group of \eqref{original} is related to that of~\eqref{eq} in all cases, and we give an algorithm to compute this relationship in \S\ref{recover}.
 
 \subsection{Organization} We will now describe the contents of the paper in some more detail. In \S\ref{preliminaries}, we will briefly recall some basic definitions and results from differential algebra, the theory of linear differential algebraic groups, Picard-Vessiot theory, and parameterized Picard-Vessiot theory. In this section, we will also define most of the notation and concepts which will be used later in the paper. In \S\ref{order1}, we will describe three auxiliary algorithms which will form the building blocks of our algorithms to compute the $\mathrm{PPV}$ group (Definition~\ref{pgal-def}). Each algorithm in this section is first followed by a concrete example of its implementation, and then by a proof that it gives the right answer. In \S\ref{order2}, we will describe algorithms to compute the $\mathrm{PPV}$ group of \eqref{original}. We will first apply Algorithms~\ref{rational} and~\ref{exponential} of \S\ref{order1} to compute the $\mathrm{PPV}$ group of an associated second order \eqref{eq}, which is a differential algebraic subgroup of $\mathrm{SL}_2$. Then, we will indicate how to recover the differential algebraic subgroup of $\mathrm{GL}_2$ corresponding to the original \eqref{original}, as an application of Algorithm~\ref{comparison} from \S\ref{order1}. In \S\ref{conclusion}, we state some consequences of the algorithms and comment on improvements which will be relevant for practical applications.


\section{Preliminaries}\label{preliminaries}

In this section, we will briefly recall without proof some standard definitions and results. We will also set the notation which we shall use for the rest of the paper. The algebraic closure of a field $K$ will be denoted by $\bar{K}$. References will be given in each section.
\vspace{-.075in}
\begin{center}\emph{Every field considered in this work will be of characteristic zero,} \\ \emph{and} $C$ \emph{will always denote an algebraically closed field.}\end{center}

\subsection{Differential algebra}\label{diffalg}

We will briefly recall some definitions from differential algebra. We refer to \cite{kaplansky:1976, kolchin:1976, vanderput-singer:2003} for what follows.

A \emph{differential field} is a pair $(K,\Delta)$, where $K$ is a field, and $\Delta$ is a finite set of pairwise commuting derivations. We abbreviate this by saying that $K$ is a $\Delta$-\emph{field}. If $\Delta=\{\delta\}$ is a singleton, we write $\delta$ instead of $\{\delta\}$, e.g. we say $K$ is a $\delta$-field. We denote the $n$-fold composition of $\delta$ with itself by $\delta^n$, with the convention that $\delta^0:=\mathrm{id}_K$. 

We say that a field automorphism $\sigma:K\rightarrow K$ is a $\Delta$\emph{-automorphism} if $\sigma\circ\delta=\delta\circ\sigma$ for every $\delta\in \Delta$. For any subset $\Pi\subseteq\Delta$, we denote by $K^\Pi$ the subset of elements $c\in K$ such that $\delta c=0$ for all $\delta\in \Pi$. One can show that $K^\Pi$ is a field, and we call it the \emph{subfield of} $\Pi$\emph{-constants} of $K$.

If $L$ is a $\Delta$-field, and $K\subseteq L$ is a subfield such that $\delta(K)\subset K$ for each $\delta\in \Delta$, we say that $L$ is a $\Delta$\emph{-field extension} of $K$, and that $K$ is a $\Delta$\emph{-subfield} of $L$. If $y_1,\dots ,y_m\in L$ are such \hypertarget{diff-gens}{that}\begin{equation*} \left\{ \delta^ny_i \ \middle| \  1\leq i \leq m, \ \delta\in \Delta,\ n\in\mathbb{Z}_{\geq 0}\right\}\end{equation*} is a set of generators for $L$ as an overfield of $K$, we write $L=K\langle y_1,\dots ,y_m\rangle_\Delta$. In this case, we also say that $L$ is \emph{differentially generated} by $\{y_1,\dots ,y_m\}$ over $K$, and that this is a set of \emph{differential generators} for $L$ over $K$.

Now let $K$ be a $\delta$-field. The \emph{ring of differential operators over} $K$ is denoted by $K[\delta]$ and consists of elements of the form $\sum_{n=0}^m r_n\delta^n$ with $r_n\in K$. The multiplication in this ring is given by composition, and is determined by the rule $\delta\circ f= f\circ\delta + \delta (f)$ for $f\in K$. The integer $m$ is the \emph{order} of the differential operator $\mathcal{D}:=\sum_{n=0}^mr_n\delta^n$, and we denote it by $\mathrm{ord}(\mathcal{D})$. We say $\mathcal{D}$ is \emph{monic} if $r_m=1$.

The \emph{ring of differential polynomials} over $K$ (in $m$ \emph{differential indeterminates}) is \begin{equation*}\smash{  K\left\{Y_1,\dots,Y_m\right\}_\delta:=K[Y_i^{(j)}]},\end{equation*} the free (commutative) $K$-algebra on the infinite set of variables $\{Y_i^{(j)}\ | \  1\leq i \leq m, j\in\mathbb{Z}_{\geq 0}\}$. We give it a $\delta$-ring structure by setting $\delta Y_i^{(j)}=Y_i^{(j+1)}$. We often omit the zero superscript, and write $Y_i=Y_i^{(0)}$. The ring $K\{Y_1,\dots , Y_m\}_\delta$ is a free object in the category of $\delta$-$K$-algebras. In particular, any ordered $m$-tuple $(y_1,\dots y_m)$ of elements of $K$ defines a unique ${\delta\text{-homomorphism}}$ $K\{ Y_1,\dots , Y_m\}_\delta\rightarrow K$, determined by $Y_i^{(j)}\mapsto \delta^jy_i$. We say that $K$ is \emph{differentially closed} if every consistent system of differential polynomial equations with coefficients in $K$ has a solution in $K$.


\subsection{Linear Differential Algebraic Groups}

The theory of differential algebraic groups was initiated by Cassidy in \cite{cassidy:1972}, and we refer to that paper for more information and complete proofs of the statements given here (see also \cite{kolchin:1984}). For the rest of this section, we assume that $K $ is a differentially closed $\partial$-field, with subfield of constants $K ^\partial=C$.

\begin{defn}\label{kolchin-top}
Let $V\subseteq K ^m$. We say $V$ is \emph{Kolchin-closed} if there are differential polynomials $p_1\dots p_r$ in $K \{ Y_1,\dots , Y_m\}_\partial$ such that \vspace{-.125in} \begin{equation*}V=\{ v\in K ^m\  |\  p_i(v)=0, \  \forall \  1\leq i \leq r\}.\end{equation*}
\end{defn}

\begin{defn}[(Cassidy {\cite[II, \S5]{cassidy:1972}})]\label{ldga-def}
A \emph{linear differential algebraic group} is a Kolchin-closed subgroup of $GL_n(K )\subset K ^{n\times n}.$
\end{defn}

\begin{thm}[(Cassidy {\cite[Prop. 11]{cassidy:1972}})]\label{dag1}
Let $B$ be a proper differential algebraic subgroup of $\mathbb{G}_a(K )$. Then, there exists a unique nonzero monic operator $ \mathcal{L}\in K [\partial]$ such that \begin{equation*}\mathbb{G}_a(K ;\mathcal{L}):=\{b\in \mathbb{G}_a(K )\  | \  \mathcal{L}b=0\}=B.\end{equation*}
\end{thm}

\begin{thm}[(Cassidy {\cite[Prop. 31 and its Corollary]{cassidy:1972}})]\label{dag2}Let $A$ be a proper differential algebraic subgroup of $\mathbb{G}_m(K )$. Then, either $A=\nobreak\mu_n\subset C^\times$, the group of $n^{\text{th}}$ roots of unity, for some $n\in\mathbb{N}$, or else there exists a unique nonzero monic operator $ \mathcal{L}\in K [\partial]$ such that \begin{equation*}\mathbb{G}_m(K ;\mathcal{L}):=\smash{\left\{a\in\mathbb{G}_m(K )\  \middle|\  \mathcal{L}\left(\tfrac{\partial a}{a}\right)=0\right\}}=A.\end{equation*}

\end{thm}

\begin{thm}[(Cassidy {\cite[Prop. 42]{cassidy:1972}})]\label{sl2}
Let $H$ be a Zariski-dense differential algebraic subgroup of $\mathrm{SL}_2(K )$. Then, either $H=\mathrm{SL}_2(K )$, or else $H$ is conjugate to $\mathrm{SL}_2(C)$.
\end{thm}


\subsection{Picard-Vessiot theory}

We refer to \cite[Ch. 1]{vanderput-singer:2003} for what follows. Let $K$ be a $\delta$-field such that $K^\delta$ is algebraically closed and let $\mathcal{D}$ be a differential operator over $K$ of order $m$.

\begin{defn}\label{pv-def}
A \emph{Picard-Vessiot extension} (or $\mathrm{PV}$\emph{-extension}) for the differential equation $\mathcal{D}Y=0$ is a $\delta$-field extension of $K$, which we denote by $\mathrm{PV}(\mathcal{D}/K)$, such that
\begin{enumerate}
\item There are $m$ distinct, $K^\delta$-linearly independent elements $y_1,\dots ,y_m\in \mathrm{PV}(\mathcal{D}/K)$ such that $\mathcal{D}y_i=0$ for all $1\leq i\leq m$.

\item $\mathrm{PV}(\mathcal{D}/K)=K\langle y_1,\dots ,y_m\rangle_\delta$.

\item $\mathrm{PV}(\mathcal{D}/K)^\delta=K^\delta$.
\end{enumerate}

\end{defn}

A $\mathrm{PV}$ extension for the operator $\mathcal{D}$ exists and is unique up to (non-unique) differential isomorphism, thus justifying the notation $\mathrm{PV}(\mathcal{D}/K)$.

\begin{defn}\label{gal-def}
The \emph{Picard-Vessiot group of} $\mathcal{D}$ \emph{over} $K$ (or $\mathrm{PV}$\emph{-group}) is the set of $\delta$-field automorphisms of $\mathrm{PV}(\mathcal{D})$ which leave $K$ point-wise fixed, and is denoted by $\mathrm{Gal}_\delta(\mathcal{D}/K)$.

The \emph{solution space} of $\mathcal{D}$ is the $K^\delta$-vector subspace of $\mathrm{PV}(\mathcal{D}/K)$ consisting of all $y\in\mathrm{PV}(\mathcal{D}/K)$ such that $\mathcal{D}y=0$, and we denote it by $\mathrm{Sol}(\mathcal{D})$.

\end{defn}

The $K^\delta$ vector space $\mathrm{Sol}(\mathcal{D})$ is $m$-dimensional, and it is stable under the action of $\mathrm{Gal}_\delta(\mathcal{D}/K)$. This action defines an injective homomorphism $\mathrm{Gal}_\delta(\mathcal{D}/K)\hookrightarrow\mathrm{GL}\smash{\bigl(\mathrm{Sol}(\mathcal{D})\bigr)}$, and one can show that its image is a linear algebraic group. There is a Galois correspondence between Zariski-closed subgroups of $\mathrm{Gal}_\delta(\mathcal{D}/K)$ and intermediate differential field extensions between $K$ and $\mathrm{PV}(\mathcal{D}/K)$ (cf. Theorem~\ref{pgal}).


\subsection{Parameterized Picard-Vessiot theory}\label{ppvtheory}

We are now ready to state the main results of the parameterized Picard-Vessiot theory. We will follow the presentation in \cite{cassidy-singer:2006}. Let $K$ be a $\Delta:=\{\delta,\partial\}$-field, such that $K_0:=K^\delta$ is a differentially closed $\partial$-field, and let $\mathcal{D}\in K[\delta]$ be a differential operator over $K$, of order $m$.

\begin{defn}\label{ppv-def} A \emph{parameterized Picard-Vessiot extension} (or $\mathrm{PPV}$\emph{extension}) for the differential equation $\mathcal{D}Y=0$ is a $\Delta$-field extension of $K$, which we denote $\mathrm{PPV}(\mathcal{D}/K)$, such that
\begin{enumerate}

\item There are $m$ distinct, $K_0$-linearly independent elements $y_1,\dots ,y_m\in\mathrm{PPV}(\mathcal{D}/K)$ such that $\mathcal{D}y_i=0$ for all $1\leq i \leq m$.

\item $\mathrm{PPV}(\mathcal{D}/K)=K\langle y_1,\dots , y_m\rangle_\Delta$.

\item $\mathrm{PPV}(\mathcal{D}/K)^{\delta}=K_0$.

\end{enumerate}

\end{defn}
 
A $\mathrm{PPV}$ extension for the operator $\mathcal{D}$ exists and is unique up to (non-unique) differential isomorphism, thus justifying the notation $\mathrm{PPV}(\mathcal{D}/K)$.
 
 \begin{defn}\label{pgal-def}
 
  The \emph{parameterized Picard-Vessiot group of} $\mathcal{D}$ \emph{over} $K$ (or $\mathrm{PPV}$ \emph{group}) is the set of $\Delta$-field automorphisms of $\mathrm{PPV}(\mathcal{D}/K)$ which leave $K$ point-wise fixed, and is denoted by $\mathrm{Gal}_\Delta(\mathcal{D}/K)$.
 
 The \emph{solution space} of $\mathcal{D}$ is the $K_0$-vector subspace of $\mathrm{PPV}(\mathcal{D}/K)$ consisting of all $y\in \mathrm{PPV}(\mathcal{D}/K)$ such that $\mathcal{D}y=0$, and we denote it by $\mathrm{Sol}(\mathcal{D})$.
 
 \end{defn}
 
One can show that $\mathrm{Sol}(\mathcal{D})$ is $m$-dimensional as a $K_0$-vector space, and the action of $\pgal(\mathcal{D}/K)$ on $\mathrm{Sol}(\mathcal{D})$ identifies $\pgal(\mathcal{D}/K)$ with a subgroup of $\mathrm{GL}\smash{\bigl(\mathrm{Sol}(\mathcal{D})\bigr)}$. It is shown in \cite{cassidy-singer:2006} that the image of $\pgal (\mathcal{D}/K)$ under this monomorphism is a linear differential algebraic group.
  
 \begin{thm}[(Parameterized Galois correspondence {\cite[Thm. 3.5]{cassidy-singer:2006}})]\label{pgal}
Let $\mathcal{D}\in K[\delta]$. Let $P$ be the set of $\Delta$-field extensions $L$ of $K$ contained in $\mathrm{PPV}(\mathcal{D}/K)$, and let $S$ be the set of Kolchin-closed subgroups $H$ of $\pgal(\mathcal{D}/K)$. 

Consider the maps $\psi:S\rightarrow P$ and $\varphi:P\rightarrow S$ defined by $\psi(H) :=\{a\in\mathrm{PPV}(\mathcal{D}/K)\  |\  \sigma a =a \  \forall \sigma\in H\}$ and $\varphi(L)  :=\pgal(\mathcal{D}/L)$. Then:
 
\begin{enumerate}

\item $\varphi$ defines an inclusion-reversing bijection $P\leftrightarrow S$, with inverse given by $\psi$.

\item $L\in P$ is a $\mathrm{PPV}$ extension of $K$ (for some operator $\mathcal{D}'$) if and only if $\varphi(L)$ is a normal subgroup of $\pgal(\mathcal{D}/K)$. In this case, $\sigma(L)=L$ for all $\sigma\in\pgal(\mathcal{D}/K)$, and the restriction homomorphism $\sigma\mapsto\sigma|_L$ is a surjection $\mathrm{Gal}_\Delta (\mathcal{D}/K) \twoheadrightarrow \mathrm{Gal}_\Delta (\mathcal{D}'/K)$. 

\end{enumerate}
 \end{thm}
 
\begin{rem}The $\delta$-subfield $K\langle y_1,\dots , y_m\rangle_{\delta}\subseteq \mathrm{PPV}(\mathcal{D}/K)$ defines a $\mathrm{PV}$ extension of $(K,\delta)$ for $\mathcal{D}$ (where $\{y_1,\dots ,y_m\}$ is any fixed $K_0$-basis of $\mathrm{Sol}(\mathcal{D})$ in $\mathrm{PPV}(\mathcal{D}/K)$). The $\delta$-field $\mathrm{PV}(\mathcal{D}/K)$ is stabilized by $\pgal(\mathcal{D}/K)$, and the resulting restriction homomorphism is in fact an injection \begin{equation*}\sigma\mapsto\sigma|_{\mathrm{PV}(\mathcal{D}/K)}: \pgal(\mathcal{D}/K)\hookrightarrow\mathrm{Gal}_\delta(\mathcal{D}/K).  \end{equation*}\end{rem}

\begin{thm}[(Cassidy-Singer {\cite[Prop. 3.6 (2)]{cassidy-singer:2006}})]\label{zariski} For $\mathcal{D}\in K[\delta_x]$, $\pgal (\mathcal{D}/K)$ is Zariski-dense in $\gal (\mathcal{D}/K)$.
 
 \end{thm}
 

 \section{Algorithms for first order equations}\label{order1}
 
In this section, we give three auxiliary algorithms which will be applied in the next section to compute the $\mathrm{PPV}$ group of a second order equation. The main idea for Algorithm~\ref{rational} is well-known (see \cite[Example 7.1]{cassidy-singer:2006}); we have merely modified it in order to avoid factoring polynomials into irreducible factors. Algorithm~\ref{rational} also helps illustrate Algorithms~\ref{exponential} and~\ref{comparison}. Algorithm~\ref{exponential} solves an analogous problem to that solved by Algorithm~\ref{rational}, but in a slightly different context. The main idea for Algorithm~\ref{exponential} comes from an argument given by Singer in the course of the proof of \cite[Prop. 4.1 (4)]{singer:2011}. Algorithm~\ref{comparison} computes the $\mathrm{PPV}$ group corresponding to the intersection of the $\mathrm{PPV}$ fields for two first order inhomogeneous equations.

\begin{setting}\hypertarget{setting}{}
For the rest of this paper we shall use the following notation: \begin{enumerate}
\item $K_0$ is a differentially closed $\partial_t$-field extension of $\smash{\bigl(C(t),\partial_t\bigr)}$, where $\partial_t$ denotes the derivation $\tfrac{d}{d t}$ on $C(t)$. 
\item $K:=K_0(x)$, endowed with the structure of $\Delta:=\{\delta_x,\partial_t\}$-field determined by setting $\delta_xx=1$, $\partial_tx=0$, and $\delta_x a=0$ for all $a\in K_0$.
\item We consider the elements of $K$ as rational functions in $x$ with coefficients in $K_0$, and we denote the \emph{degree} of a polynomial in $x$ and \emph{greatest common divisor} of polynomials in $x$ by $\mathrm{deg}_x$ and $\mathrm{gcd}_x$, respectively. We adopt the convention that $\mathrm{deg}_x(0)=0$.
\end{enumerate}
\end{setting}

Consider the $\Delta$-field extension $F$ of $K$ of given by $F:=K(\eta)$, where $\delta_x\eta:=p\eta$, $\partial_t\eta:=q\eta$, and $p,q\in \overline{C(t)}(x)$ satisfy the \emph{integrability condition} $\delta_xq=\partial_tp$. We do not exclude the possibility that $\eta\in K$ or, equivalently, that $F=K$. We wish to compute the $\mathrm{PPV}$ group corresponding to the first order (inhomogeneous) differential equation $\delta_xY=\eta$ over $F$. 

We claim this is a differential algebraic subgroup of $\mathbb{G}_a(K_0)$. To see this, consider the associated homogeneous equation \cite[Ex. 1.18]{vanderput-singer:2003}: let $\mathcal{P}:=\delta_x^2-p\delta_x$. One can show that, if $\{1,\theta\}$ is a $K_0$-basis of $\mathrm{Sol}(\mathcal{P})$, then $\mathrm{PPV}(\mathcal{P}/F)=F\langle\theta\rangle_\Delta$. We may assume without loss of generality that $\delta_x\theta=\eta$, because $\delta_x(\delta_x\theta)=p\delta_x\theta$ and $\delta_x\theta\neq 0$ together imply that $\delta_x\theta =a\eta$ for some $a\in K_0^\times$. It follows from this that $\delta_x(\sigma \theta - \theta)=0$ for all $\sigma\in\pgal(\mathcal{P}/F)$. The map $\sigma\mapsto b_\sigma:=(\sigma\theta-\theta)$ defines an isomorphism of $\pgal(\mathcal{P}/F)$ onto a differential algebraic subgroup of $\mathbb{G}_a(K_0)$. From now on, we will identify $\pgal(\mathcal{P}/F)$ with its image in $\mathbb{G}_a(K_0)$ via this isomorphism.

Simply put, each algorithm will use the input data to construct a system of homogeneous linear equations, and then construct the output from a solution for the system which satisfies a certain minimality condition. The construction of these systems of equations will rely on the following Lemma.

\begin{lem}[(Reduction Lemma)]\label{singer1}

Let $p,q\in K$, $\eta\in F$, and $\mathcal{P}\in K[\delta_x]$ be defined as in the preceding discussion, and let $0\neq\mathcal{L}\in K_0[\partial_t]$. The following are equivalent:

\begin{enumerate}

\item \label{gcond} $\pgal(\mathcal{P}/F)\subseteq\mathbb{G}_a(K_0;\mathcal{L}).$ 

\item \label{rcond}$\mathcal{L}(\eta)=\delta_xf$ for some $f\in F$.

\item \label{tcond} $\mathcal{L}(\eta)=(\delta_xh +ph)\eta$ for some $h\in K$.

\end{enumerate}

\end{lem}
 
 \begin{proof}
 
Assume that $\pgal(\mathcal{P}/F)\subseteq\mathbb{G}_a(K_0;\mathcal{L}),$ or, equivalently, that $\mathcal{L}(b_\sigma)=0$ for all $\sigma\in\pgal(\mathcal{P}/F)$. Since \begin{equation*}\sigma\left(\mathcal{L}(\theta)\right)=\mathcal{L}(\theta)+\mathcal{L}(b_\sigma)=\mathcal{L}(\theta),\end{equation*} we have that $\mathcal{L}(\theta)\in F$ by the Galois correspondence. Since $\delta_x\mathcal{L}(\theta)=\mathcal{L}(\eta)$, setting $f:=\mathcal{L}(\theta)$ gives the implication \eqref{gcond} $\Rightarrow$ \eqref{rcond}. Conversely, assume that there exists $ f\in F$ such that $\delta_xf=\mathcal{L}(\eta)$. Then $\delta_x(\mathcal{L}(\theta)-f)=0$, so $\mathcal{L}(\theta)=f+a$ for some $a\in K_0$. This implies that $\mathcal{L}(b_\sigma)=0$ for all $\sigma$, and we have shown that \eqref{rcond} $\Rightarrow$ \eqref{gcond}.

To see that \eqref{tcond} $\Rightarrow$ \eqref{rcond}, let $h\in K$ be as in part \eqref{tcond} of the Lemma, and set $f:=h\eta$.

Now we will show that \eqref{rcond} $\Rightarrow$ \eqref{tcond}. Since $\frac{\partial_t\eta}{\eta}\in K$, we have that $\smash{\frac{\partial_t^n\eta}{\eta}}\in K$ for any $n$, so $\smash{\frac{\mathcal{L}(\eta)}{\eta}}\in K$. If $\eta\in K$, setting $h:=\smash{\frac{f}{\eta}}$ establishes our claim. If $\eta\notin K$, we have two cases to consider: either $\eta$ is algebraic over $K$ or $\eta$ is transcendental over $K$.

If $\eta$ is algebraic over $K$ of degree $m$, then $\{1,\eta,\dots , \eta^{m-1}\}$ is a basis for $F$ as a $K$-vector space, so we may write $f=\sum_{i=0}^{m-1}a_i\eta^i$ with $a_i\in K$ in a unique way. Then $\delta_xf=\sum(\delta_xa_i+ipa_i)\eta^i$ and, since $\frac{\mathcal{L}(\eta)}{\eta}=\frac{\delta_xf}{\eta}\in K$, we have that \vspace{-.12in}\begin{equation*} \frac{\delta_xf}{\eta}=\sum_{i=0}^{m-1}(\delta_xa_i+ipa_i)\eta^{i-1}=\delta_xa_1+pa_1, \end{equation*} because the set $\{\eta^{-1},1,\dots , \eta^{m-2}\}$ is also a basis for $F$ as a $K$-vector space. Therefore, $\delta_xf=\delta_x(a_1\eta)$, and setting $h:=a_1$ establishes our claim in this case.

Now assume that $\eta$ is transcendental over $K$. Then $f$ has a unique partial fraction decomposition:\begin{equation*} \smash{f=\sum_i a_i\eta^i+\sum_{j,k}\frac{c_{j,k}}{(\eta-e_k)^j}},\end{equation*} where $a_i\in K$ and $c_{j,k},e_k\in \bar{K}$, because the set \begin{equation*} \smash{\mathcal{B}:=\{\eta^i\ |\ i\geq 0\}\cup\bigcup_{e\in \bar{K}}\left\{\tfrac{1}{(\eta-e)^j}\ \middle|\ j\geq 1\right\}}\vphantom{\sum_i}\end{equation*} is a basis for $\bar{K}(\eta)$ as a $\bar{K}$-vector space. Since the image of $\mathcal{B}$ under the invertible $\bar{K}$-linear map given by multiplication by $\eta^{-1}$ is still a basis, and $\frac{\mathcal{L}(\eta)}{\eta}=\frac{\delta_xf}{\eta}\in K$, we see that \begin{align*} \frac{\delta_xf}{\eta} & = \sum_i (\delta_xa_i +ipa_i)\eta^{i-1}+\sum_{j,k} \frac{\delta_xc_{j,k}(\eta-e_k)-jc_{j,k}(p\eta-\delta_xe_k)}{\eta(\eta-e_k)^{j+1}} \\ & =  \sum_i (\delta_xa_i +ipa_i)\eta^{i-1}+\sum_{j,k} \frac{\delta_xc_{k,j}-jpc_{j,k}}{\eta(\eta-e_k)^j}+\frac{jc_{j,k}(\delta_xe_k-pe_k)}{\eta(\eta-e_k)^{j+1}}  \\ &=\delta_xa_1+pa_1. \end{align*}Therefore, $\delta_xf=\delta_x(a_1\eta)$, and setting $h:=a_1$ establishes the Lemma. \end{proof}
 
 \begin{rem}
The equivalence of \eqref{gcond} and \eqref{rcond} in the Lemma is well-known (see \cite{singer:2011, dreyfus:2011}). The fact that \eqref{tcond} $\Rightarrow$ \eqref{rcond} was used by Singer in \cite{singer:2011}. We do not know of a proof of the implication \eqref{rcond} $\Rightarrow$ \eqref{tcond} in the literature.
\end{rem}

The following two algorithms compute the operator $\mathcal{L}\in K_0[\partial_t]$ such that $\pgal(\mathcal{P}/F)\simeq\mathbb{G}_a(K_0;\mathcal{L})$. Since $\mathbb{G}_a(K_0;\mathcal{L})=\mathrm{Sol}(\mathcal{L})$, its dimension as a $C$-vector space is equal to $\mathrm{ord}(\mathcal{L})$, and $\mathcal{L}$ is the operator of smallest order such that the equivalent conditions of Lemma~\ref{singer1} are satisfied. Algorithm~\ref{rational} applies only when $\eta\in K$, while Algorithm~\ref{exponential} applies only when $\eta\notin K$.

\subsection{First case: $\eta\in K$} By part \eqref{rcond} of Lemma~\ref{singer1}, the operator ${\mathcal{L}\in K_0[\partial_t]}$ such that $\pgal(\mathcal{P}/K)\simeq\mathbb{G}_a(K_0;\mathcal{L})$ is the operator of smallest order such that $\mathcal{L}\eta=\delta_xf$, for some $f\in K=F$. We will write $\mathcal{L}$ and $f$ with undetermined coefficients, and use this relation to obtain a system of linear equations in the coefficients of $\mathcal{L}$ and~$f$. 


\begin{algo}[(Primitive of a rational)]\label{rational}$\hphantom{X}$
\begin{center}
\begin{boxedminipage}{11.8cm}\textbf{Input:} $\eta\in \overline{C(t)}(x)$.

\hypertarget{output1}{\textbf{Output:}} $\mathcal{L}^\eta:=\sum_{i=0}^\ell a_i\partial_t^i\in\overline{C(t)}[\partial_t]$  such  that  $\pgal(\mathcal{P}/K)\simeq\mathbb{G}_a(K_0;\mathcal{L}^\eta)$.\end{boxedminipage}
\end{center}

\begin{asparaenum}[\bfseries {Step} 1:] 

\item \label{step1.1}Write \vspace{-.125in}\begin{equation*} \eta=\eta_0+\frac{\eta_1}{\eta_2}, \end{equation*} with $\eta_i\in\smash{\overline{C(t)}}[x]$ for $0\leq i \leq 2$, such that \begin{equation*} \mathrm{gcd}_x(\eta_1,\eta_2)=1\quad\text{and}\quad \mathrm{deg}_x(\eta_1)<\mathrm{deg}_x(\eta_2).\end{equation*} \vspace{-.1in}

\item \label{step1.2}Let \vspace{-.125in}\begin{equation*} d:=\frac{\eta_2}{\mathrm{gcd}_x(\eta_2,\delta_x\eta_2)}, \end{equation*} and let $n$ be the smallest integer such that $d^n\eta$ is a polynomial in $x$. We remark that $d$ is the product of the irreducible factors of $\eta_2$, and $n$ is the highest multiplicity of an irreducible factor of $\eta_2$.

\item \label{step1.3}Let $M:=\mathrm{deg}_x(d)$ and $s:=\mathrm{deg}_x(\eta_0)$. For each $0\leq N\leq M$, write \begin{equation*} \mathcal{L}_N  :=\sum_{i=0}^N\alpha_i\partial_t^i \qquad\text{and}\qquad f_N  :=\sum_{j=0}^{s+1}\beta_jx^j+\sum_{k=1}^{(n+N-1)}\left(\frac{\sum_{l=0}^{M-1}\xi_{k,l}x^l}{d^k}\right),\end{equation*} where the $\alpha_i$, $\beta_j$, and $\xi_{k,l}$ are undetermined coefficients.

\item \label{step1.4}Treating each $\beta_j$ and $\xi_{k,l}$ as a $\delta_x$-constant, set \begin{equation*}\tag{$\mathbf{H}_N$} \label{nrational} \delta_xf_N=\mathcal{L}_N(\eta), \end{equation*} and then multiply each side of this equation by $d^{n+N}$. The result will be an equality of polynomials in $x$, whose coefficients are homogeneous linear forms in the $\alpha_i$, $\beta_j$, and $\xi_{k,l}$. Equating coefficients of like-powers of $x$, we obtain a system of \vspace{-.125in}\begin{equation*} M(n+N)+s \end{equation*} homogeneous linear equations with coefficients in $\overline{C(t)}$, in the \vspace{-.06in}\begin{equation*} (N+1) +(s+2)+M(n+N-1)\vspace{-.12in} \end{equation*}variables $\alpha_i$, $\beta_j$ and $\xi_{k,l}$.

\item \label{step1.5}If $N=M$, the system of linear equations defined by~\eqref{nrational} has a solution with not all $\alpha_i$ being equal to zero. Find the smallest nonnegative integer $\ell\leq M$ such that the system of linear equations defined by~\eqref{nrational}, with $N=\ell$, has a solution with not all $\alpha_i$ being zero.

\item \label{step1.6}Find a solution \begin{align*}  \alpha_i&=a_i \quad &&\text{for}\ 0\leq i \leq \ell, \\ \beta_j&=b_j\quad &&\text{for}\ 0\leq j \leq s+1,\\ \xi_{k,l}&=c_{k,l} \quad &&\text{for}\ 1\leq k \leq n+N-1\ \text{and}\  0 \leq l \leq M-1;\end{align*}for this system of homogeneous linear equations, such that each $a_i,b_j,c_{k,l}\in\overline{C(t)}$ and $a_\ell=1$. This condition determines $a_0,\dots,a_\ell$ uniquely.

\item Set $\mathcal{L}^\eta:=\sum_{i=0}^\ell a_i\partial_t^i$, and go to \hyperlink{output1}{Output}.

\end{asparaenum}
\end{algo}


\begin{rem}
Upon inspection, we see that the system of linear equations of Step~\ref{step1.4} has coefficients in the smallest algebraic extension $L$ of $C(t)$ such that $\eta\in L(x)$. Therefore, $\mathcal{L}^\eta\in L[\partial_t]$.
\end{rem}

 \begin{rem}
 In practice, it is often possible to calculate $\mathcal{L}^\eta$ more directly, as follows: let $\{d_1,\dots ,d_M\}$ be the set of poles of $\eta$, and let $n$ be the maximum order of $\eta$ at any of these poles. If \begin{equation*} \eta=\sum_{j=0}^n b_jx^j + \sum_{k=1}^n \sum_{l=1}^{M}\frac{e_{k,l}}{(x-d_l)^k} \end{equation*} is the partial fraction decomposition of $\eta$, it is shown in \cite[Example 7.1]{cassidy-singer:2006} that $\mathbb{G}_a(K_0;\mathcal{L}^\eta)$ is the $C$-span of $\{ e_{1,l}\}_{i=1}^{M}$ in $K_0$. Incidentally, this proves the claim made in Step~\ref{step1.5}, that the system obtained from~\eqref{nrational} with $N=M$ has a solution with not all $\alpha_i$ being 0, because \begin{equation*} M\geq\mathrm{dim}_C\big(\mathbb{G}_a(K_0;\mathcal{L}^\eta)\big)=\mathrm{ord}(\mathcal{L}^\eta). \end{equation*} 
If $\{\tilde{e}_k\}_{k=1}^\ell$ is a basis of $\mathbb{G}_a(K_0;\mathcal{L}^\eta)$ considered as a $C$-vector space, then $\mathcal{L}^\eta$ is given by $\sum_{i=0}^\ell a_i\partial_t^i$, where \begin{equation*}\mathrm{det}\begin{pmatrix} Y^{(0)} & Y^{(1)} & \cdots & Y^{(\ell)} \\ \tilde{e}_1& \partial_t\tilde{e}_1 & \cdots & \partial_t^\ell \tilde{e}_1 \\ \vdots & \vdots &\vdots &\vdots \\  \tilde{e}_\ell &\partial_t \tilde{e}_\ell &\cdots & \partial_t^\ell \tilde{e}_\ell\end{pmatrix}=\sum_{i=0}^\ell a_iY^{(i)}.  \end{equation*} Algorithm~\ref{rational} circumvents the need to factorize the denominator of $\eta$, because performing such a factorization could be computationally infeasible in practice. However, when such a factorization is available, it should be more efficient to carry out this simpler algorithm, instead of Algorithm~\ref{rational}.
 \end{rem}

\begin{eg}\label{example1}We will now apply Algorithm~\ref{rational} with $\eta:=\frac{2t}{x^2+t},$ to compute the $\mathrm{PPV}$ group corresponding to \begin{equation} \label{eg-1-eq1}\partial Y/\partial x= 2t/(x^2+t).\end{equation} In this case, $d=x^2+t$, $n=1$, $M=2$, and $s=0$. We write the operator $\mathcal{L}_N$ and the rational function $f_N$ with undetermined coefficients, as in Step~\ref{step1.3}, with $N=1$:\vspace{-.1in}\begin{equation*} \mathcal{L}_1:=\alpha_1\tfrac{\partial}{\partial t}+\alpha_0 \qquad\text{and}\qquad f_1:=\beta_0+\beta_1x+\frac{\xi_{1,0}+\xi_{1,1}x}{x^2+t}.\end{equation*} Step~\ref{step1.4} then requires us to substitute these expressions in \eqref{nrational}, with $N=1$, to obtain\begin{equation}\label{eg-1-eq2}\frac{\partial f_1}{\partial x} =\beta_1+\frac{\xi_{1,1}}{x^2+t}+\frac{-2\xi_{1,0}x-2\xi_{1,1}x^2}{(x^2+t)^2}  = \frac{2t\alpha_0+2\alpha_1}{x^2+t}+\frac{-2t\alpha_1}{(x^2+t)^2}=\mathcal{L}_1(\eta).\end{equation} After multiplying by $d^2=(x^2+t)^2$ on both sides, we have that~\eqref{eg-1-eq2} holds if and only if the following polynomial in $x$ is zero\vspace{-.125in}\begin{equation*}  (-\beta_1)x^4+\bigl(2t\alpha_0+2\alpha_1-2t\beta_1+\xi_{1,1}\bigr)x^2+(2\xi_{1,0})x+\bigl( 2t^2\alpha_0-t^2\beta_1 -t\xi_{1,1} \bigr). \end{equation*} Thus we obtain a system of linear equations by setting each coefficient equal to zero. We then find the solution $\alpha_0=-\frac{1}{2t}$, $\alpha_1=1$, $\beta_1=\beta_0=\xi_{1,0}=0$, $\xi_{1,1}=-1$, and we thus obtain an operator $\mathcal{L}\in C(t)[\frac{\partial}{\partial t}]$ and a rational function $f\in C(x,t)$, defined as $\mathcal{L}:=\smash{\frac{\partial}{\partial t}-\frac{1}{2t}}$ and $\smash{f:=\frac{-x}{x^2+t}}$, which satisfy condition \eqref{rcond} of Lemma~\ref{singer1}. One can check that there is no solution with $\alpha_1=0$, and therefore the $\mathrm{PPV}$ group of \eqref{eg-1-eq1} is $\mathbb{G}_a(K_0;\frac{\partial}{\partial t}-\frac{1}{2t})$. 
\end{eg}

\begin{prop}[(Algorithm~\ref*{rational} is correct)] \label{alg1}Suppose that $\eta\in K$. Then, the \hyperlink{output1}{Output} of Algorithm~\ref{rational} is the operator which defines $\pgal(\mathcal{P}/K)$ as a subgroup of $\mathbb{G}_a(K_0)$. In other words, $\pgal(\mathcal{P}/K)\simeq\mathbb{G}_a(K_0;\mathcal{L}^\eta)$.
\end{prop}

\begin{proof}The solution found in Step~\ref{step1.6}, for the system of linear equations of Step~\ref{step1.4}, also defines a rational function \begin{equation*} f:= \sum_{j=0}^{s+1}b_jx^j+\sum_{k=1}^{(n+\ell-1)}\left(\frac{\sum_{l=0}^{M-1}c_{k,l}x^l}{d^k}\right).\end{equation*} By construction, $\mathcal{L}^\eta\in K_0[\partial_t]$ and $f\in K$ satisfy part \eqref{rcond} of Lemma~\ref{singer1}. Therefore, $\pgal(\mathcal{P}/K)\subseteq\mathbb{G}_a(K_0;\mathcal{L}^\eta)$.

We need to rule out the possibility that there exist $\smash{\tilde{\mathcal{L}}}\in K_0[\partial_t]$ and $\smash{\tilde{f}}\in K$ satisfying part \eqref{rcond} of Lemma~\ref{singer1} with $\smash{\tilde{\ell}}:=\mathrm{ord}(\smash{\tilde{\mathcal{L}}})<\mathrm{ord}(\mathcal{L}^\eta)=\ell$. We proceed by contradiction: assume that such $\smash{\tilde{\mathcal{L}}}$ and $\smash{\tilde{f}}$ do exist. Write \begin{equation*}\tilde{\mathcal{L}}:=\sum_{i=0}^{\tilde{\ell}}\tilde{a}_i\partial_t^i,\qquad\text{and}\qquad \tilde{f}  :=\sum_{j=0}^{s+1}\tilde{b}_jx^j+\sum_{k=1}^{T}\left(\frac{\sum_{l=0}^{M-1}\tilde{c}_{k,l}x^l}{d^k}\right),\end{equation*} for some $T\in\mathbb{N}$, with $\tilde{a}_i,\tilde{b}_j,\tilde{c}_{k,l}\in K_0$. A calculation shows that $d^{n+\smash{\tilde{\ell}}}\smash{\tilde{\mathcal{L}}}(\eta)\in K_0[x]$. Therefore, $\smash{d^{n+\tilde{\ell}}}\delta_x\smash{\tilde{f}}\in K_0[x]$, and we may take $T=n+\smash{\tilde{\ell}}-1$. But this means that $\alpha_i=\smash{\tilde{a}}_i$, $\beta_j=\smash{\tilde{b}}_j$, $\xi_{k,l}=\smash{\tilde{c}}_{k,l}$ is a solution for the system of equations defined by~\eqref{nrational} with $N=\smash{\tilde{\ell}}$, which contradicts the choice~of~$\ell$. \end{proof}

\subsection{Second case: $\eta\notin K$}  By part \eqref{tcond} of Lemma~\ref{singer1}, the desired $\mathcal{L}\in K_0[\partial_t]$ is the operator of smallest order such that $\mathcal{L}(\eta)=(\delta_xh+ph)\eta$ for some $h\in K$. We will write $\mathcal{L}$ and $h$ with undetermined coefficients, and use this relation to obtain a system of linear equations in the unknown coefficients of $\mathcal{L}$ and~$h$.


\begin{algo}[(Primitive of an exponential)]\label{exponential} $\hphantom{X}$

\begin{center}
\begin{boxedminipage}{12.4cm}\textbf{Input:} $p,q\in \overline{C(t)}(x)$ such that $\partial_tp=\delta_xq$.

\hypertarget{output2}{\textbf{Output:}} $\mathcal{L}^{p,q}:=\sum_{i=0}^\ell a_i\partial_t^i\in\overline{C(t)}[\partial_t]$  such  that  $\pgal(\mathcal{P}/F)\simeq\mathbb{G}_a(K_0;\mathcal{L}^{p,q})$.\end{boxedminipage}
\end{center}

\begin{asparaenum}[\bfseries {Step} 1:]

\item \label{step2.1}Write \vspace{-.125in}\begin{equation*} \vspace{-.07in}p=p_0+\frac{p_1}{p_2}\quad \text{and} \quad q=q_0+\frac{q_1}{q_2},\end{equation*} with $p_k,q_k\in \overline{C(t)}[x]$ for $0\leq k \leq 2$, such that \begin{gather*}  \mathrm{gcd}_x(p_1,p_2)=1=\mathrm{gcd}_x(q_1,q_2), \\ \mathrm{deg}_x(p_1)<\mathrm{deg}_x(p_2) \quad \text{and}\quad\mathrm{deg}_x(q_1)<\mathrm{deg}_x(q_2).\end{gather*}

\item \label{step2.2}Let \vspace{-.125in}\begin{equation*} \smash{d:=\frac{p_2q_2}{\mathrm{gcd}_x\bigl(p_2q_2,\delta_x(p_2q_2)\bigr)}},\vphantom{\sum_i} \end{equation*} let $\nu$ be the smallest integer such that $d^\nu p$ and $d^\nu q$ are both polynomials in $x$, and define \begin{equation*} m  := \mathrm{deg}_x(d),\quad  n  := \mathrm{max}\{\mathrm{deg}_x(p_0),\mathrm{deg}_x(q_0), \nu\},\quad\text{and}\quad  M  := mn+n+1.\end{equation*} We remark that $d$ is the product of the irreducible factors of $p_2q_2$, and $\nu$ is the highest multiplicity of an irreducible factor of $p_2$ or $q_2$.

 \item \label{step2.3}Define a sequence $\{R_i\}\subset K$ recursively: $R_0  :=1$ and $R_i  :=\partial_tR_{i-1}+qR_{i-1}$ for $i\geq 1$.

\item \label{step2.4}For each nonnegative integer $N\leq M$, consider the inhomogeneous first order differential equation (with undetermined coefficients $\alpha_0,\dots ,\alpha_N$): \vspace{-.13in}\begin{equation}\label{ihomy} \delta_xY+  pY=\sum_{i=0}^N\alpha_iR_i.\end{equation} Find bounds\footnote{The computation of these bounds is elementary, but somewhat lengthy. We will indicate how to compute these bounds in Remark \ref{compute-bounds}.} $S_N, T_N\in\mathbb{N}$, depending only on $p$, $q$, and $N$, such that the following condition holds: if $a_0,\dots,a_N\in K_0$ is \emph{any} $(N+1)$-tuple of elements, and $h_N\in K$ satisfies the equation obtained from \eqref{ihomy} after replacing $\alpha_i$ by $a_i$ for each $0\leq i\leq N$, then $h_N$ must be of the form \begin{equation}\label{rh}  h_N=\sum_{j=0}^{S_N}\beta_jx^j+\sum_{k=1}^{T_N}\left(\frac{\sum_{l=0}^{m-1}\xi_{k,l}x^l}{d^k}\right), \end{equation} for some $\beta_0 , \dots ,\beta_{S_N},\xi_{1,0},\dots , \xi_{T_N,m-1}\in K_0$. That is, the degree of the polynomial part of $h_N$ is bounded by $S_N$, and the order of any pole of $h_N$ is bounded by $T_N$, regardless of the values of $\alpha_0,\dots,\alpha_N$ in $K_0$.

\item \label{step2.5}Replace $Y$ with this expression for $h_N$ in \eqref{ihomy}, to obtain \begin{equation*}\tag{$\mathbf{I}_N$}\label{ihomh} \delta_xh_N+ph_N=\sum_{i=0}^N\alpha_iR_i,\end{equation*} where the $\beta_j$ and $\xi_{k,l}$ are to be treated as undetermined $\delta_x$-constants. Multiplying by $d^{T_N+n}$ on each side, we obtain an equality of polynomials in $x$, whose coefficients are homogeneous linear forms in the $\alpha_i$, $\beta_j$, and $\xi_{k,l}$. Equating like-powers of $x$ on both sides, we obtain a system of \begin{equation*}m(T_N+n)+ S_N+n \end{equation*}homogeneous linear equations with coefficients in $\overline{C(t)}$, in the \begin{equation*} (N+1)+(S_N+1)+mT_N \vspace{-.08in}  \end{equation*}variables $\alpha_i$, $\beta_j$, and $\xi_{k,l}$.

\item \label{step2.6}If $N=M$, the system of linear equations defined by~\eqref{ihomh} has a solution. Find the smallest nonnegative integer $\ell\leq M$ such that the system of linear equations defined by~\eqref{ihomh}, with $N=\ell$, has a solution.

\item \label{step2.7}Find a solution \begin{align*}  \alpha_i&=a_i \quad &&\text{for}\ 0\leq i \leq \ell, \\ \beta_j&=b_j\quad &&\text{for}\ 0\leq j \leq S_\ell,\\ \xi_{k,l}&=c_{k,l} \quad &&\text{for}\ 1\leq k \leq T_N\ \text{and}\  0 \leq l \leq m-1;\end{align*} for this system of homogeneous linear equations, such that each $a_i,b_j,c_{k,l}\in\overline{C(t)}$ and $a_\ell=1$. This condition determines $a_0,\dots ,a_\ell$ uniquely.

\item Set $\mathcal{L}^{p,q}:=\sum_{i=0}^\ell a_i\partial_t^i$, and go to \hyperlink{output2}{Output}.

\end{asparaenum}
\end{algo}


\begin{rem}
Upon inspection, we see that the system of linear equations of Step~\ref{step2.5} has coefficients in the smallest algebraic extension $L$ of $C(t)$ such that $p,q\in L(x)$. Therefore, $\mathcal{L}^{p,q}\in L[\partial_t]$.\end{rem}

\begin{rem}As in \cite[Prop. 4.1(4)]{singer:2011}, the operator $\mathcal{L}^{p,q}$ produced by Algorithm~\ref{exponential} must be nonzero. Otherwise, the element $h\in\overline{C(t)}(x)$ obtained from \eqref{rh} by substituting $\beta_j=b_j$ and $\xi_{k,l}=c_{k,l}$ would have to satisfy the \eqref{ihomh} with $N=0$ and $\alpha_0=0$, which is $\delta_xh=ph$. This implies that $h=a\eta$ for some $a\in K_0$, which would contradict the assumption that $\eta\notin K$.\end{rem}

\begin{eg}[(The Picard-Fuchs equation for the Legendre family of elliptic curves)] Consider the first order inhomogeneous equation\vspace{-.125in} \begin{equation} \label{eg-2-eq1} \partial Y/\partial x=\smash{\bigl(x(x-1)(x-t)\bigr)^{-\frac{1}{2}}} .\end{equation} We will take $\eta:=\bigl(x(x-1)(x-t)\bigr)^{-\frac{1}{2}}$, and apply Algorithm~\ref{exponential} with inputs \begin{equation*}  p:= \frac{\partial\eta/\partial x}{\eta}=-\frac{1}{2}\left(\frac{1}{x}+\frac{1}{x-1}+\frac{1}{x-t}\right)\qquad\text{and}\qquad q:=\frac{\partial\eta/\partial t}{\eta}=\frac{1}{2}\left(\frac{1}{x-t}\right),\end{equation*} to derive the Picard-Fuchs equation for the Legendre family of elliptic curves $E_t:y^2=x(x-1)(x-t)$ (cf. \cite[pp. 77-78]{manin:1964} and \cite[Example 6.9]{gor-ov:2012}). 

In Step~\ref{step2.2}, we compute $d=x(x-1)(x-t)$, $\nu=1$, $m=3$, $n=1$, and $M=5$. Now let \begin{equation*} R_0=1,\quad R_1=q=  \frac{1}{2}\left(\frac{1}{x-t}\right),\quad R_2=\frac{\partial q}{\partial t} +q^2= \frac{3}{4}\left(\frac{1}{(x-t)^2}\right),\quad R_3=\cdots\end{equation*} be the sequence defined in Step~\ref{step2.3} (we will have no need to compute $R_i$ for $i>2$), and consider the inhomogeneous equation with undetermined coefficients of Step~\ref{step2.4}, with $N=2$\begin{equation*}\frac{\partial Y}{\partial x}-\frac{1}{2}\left(\frac{1}{x}+\frac{1}{x-1}+\frac{1}{x-t}\right)Y= \alpha_0+\frac{\alpha_1}{2}\left(\frac{1}{x-t}\right)+\frac{3\alpha_2}{4}\left(\frac{1}{(x-t)^2}\right), \end{equation*} A computation shows that we may take $T_2=1$ and $S_2=1$ (see Remark \ref{compute-bounds}, or the proof of Proposition~\ref{borel-proof} below). Now we write the rational function $h_2=\beta_0+\beta_1x+\frac{\xi_{0,1}}{x-t}$ with undetermined coefficients $\beta_0,\beta_1,\xi_{0,1}$ as in Step~\ref{step2.4}, to obtain \eqref{ihomh} with $N=2$ as in Step~\ref{step2.5}: \begin{equation}\label{eg2-ihom}\tfrac{\partial}{\partial x} \left(\beta_0+\beta_1x+\tfrac{\xi_{1,0}}{x-t}\right) -\tfrac{1}{2}\left(\tfrac{1}{x}+\tfrac{1}{x-1}+\tfrac{1}{x-t}\right)\left(\beta_0+\beta_1x+\tfrac{\xi_{1,0}}{x-t}\right)  = \alpha_0+\tfrac{\alpha_1}{2}\left(\tfrac{1}{x-t}\right)+\tfrac{3\alpha_2}{4}\left(\tfrac{1}{(x-t)^2}\right). \end{equation} After multiplying by $x(x-1)(x-t)^2$ to clear denominators,\footnote{Note that here the algorithm calls for multiplying by $d^2=\bigl(x(x-1)(x-t)\bigr)^2$, but in this concrete example we see directly that this is unnecessary. One can use square-free factorizations to do this in general (see Section \ref{conclusion}).}  and then subtracting the left-hand side from the right-hand side, we obtain that~\eqref{eg2-ihom} holds if and only if the following polynomial in $x$ is zero:\begin{gather*}  (\alpha_0+\tfrac{1}{2}\beta_1)  x^4  +\bigl( (-2t-1)\alpha_0-\tfrac{1}{2}\alpha_1+\tfrac{3}{2}\beta_0-\tfrac{1}{2}t\beta_1  \bigr)  x^3  \\  +\bigl(  (t^2+2t)\alpha_0+(\tfrac{1}{2}t+\tfrac{1}{2})\alpha_1-\tfrac{3}{4}\alpha_2- (\tfrac{5}{2}t+1)\beta_0-\tfrac{1}{2}t\beta_1+\tfrac{5}{2}\xi_{1,0} \bigr) x^2 \\ +\bigl(-t^2\alpha_0-\tfrac{1}{2}t\alpha_1+\tfrac{3}{4}\alpha_2+(t^2+\tfrac{3}{2}t)\beta_0+\tfrac{1}{2}t^2\beta_1-(t+2)\xi_{1,0}      \bigr)   x  +(-\tfrac{1}{2}t^2\beta_0+\tfrac{1}{2}t\xi_{1,0}) x^0 .\end{gather*} Setting each coefficient equal to zero yields the system of linear equations defined in Step~\ref{step2.5}. One can check that $\alpha_0=-\frac{1}{4}$, $\alpha_1=2t-1$, $\alpha_2=t(t-1)$, $\beta_0=\frac{1}{2}(t-1)$, $\beta_1=\frac{1}{2}$, and $\xi_{1,0}=\frac{1}{2}t(t-1)$ is a solution,\footnote{We have taken $\alpha_2=t(t-1)$ instead of $\alpha_2=1$, because this is how the Picard-Fuchs equation usually appears in the literature. } and that indeed\begin{equation*} t(t-1)\frac{\partial^2 \eta}{\partial t^2} +(2t-1)\frac{\partial\eta}{\partial t}-\frac{1}{4}\eta=\frac{\partial}{\partial x}\left( \left(\frac{1}{2}(t-1)+\frac{1}{2}x^2+\frac{\tfrac{1}{2}t(t-1)}{x-t} \right)\eta \right).\end{equation*} After checking that the system of linear equations does not have a solution with $\alpha_2=0$, we conclude that the $\mathrm{PPV}$ group corresponding to \eqref{eg-2-eq1} is $\mathbb{G}_a\bigl(K_0; \ t(t-1)\frac{\partial^2 }{\partial t^2} +(2t-1)\frac{\partial}{\partial t}-\frac{1}{4}\bigr)$.\end{eg}

\begin{prop}[(Algorithm~\ref*{exponential} is correct)]\label{alg2}
Suppose that $\eta\notin K$. Then, the \hyperlink{output2}{Output} of Algorithm~\ref{exponential} is the operator which defines $\pgal(\mathcal{P}/F)$ as a subgroup of $\mathbb{G}_a(K_0)$. In other words, $\pgal(\mathcal{P}/F)\simeq\mathbb{G}_a(K_0;\mathcal{L}^{p,q})$.
\end{prop}

\begin{proof} Consider the sequence $\{R_i\}\subset K$ defined in Step~\ref{step2.3}. Since $\partial_t\eta=q\eta$, we have that $\partial_t^i\eta=R_i\eta$ for each $i\geq 0$, and therefore \begin{equation*} \smash{\sum_{i=0}^N\alpha_i\partial_t^i\eta=\left(\sum_{i=0}^N\alpha_iR_i\right)\eta},\vphantom{\sum_i}  \end{equation*} for each $0\leq N\leq M$. Note also that the solution found in Step~\ref{step2.7}, for the system of polynomial equations defined in Step~\ref{step2.5}, also defines a rational function $h\in\overline{C(t)}[x]$, obtained from the rational function (with undetermined coefficients) $h_\ell$ of \eqref{rh} by setting $\beta_j=b_j$ and $\xi_{k,l}=c_{k,l}$ for each $j,k,l$: \begin{equation*} h:= \sum_{j=0}^{S_\ell}b_jx^j+\sum_{k=1}^{T_\ell}\left(\frac{\sum_{l=0}^{m-1}c_{k,l}x^l}{d^k}\right)\in\overline{C(t)}(x).\end{equation*} By construction, $\mathcal{L}^{p,q}\in K_0[\partial_t]$ and $h_\ell\in K$ satisfy part \eqref{tcond} of Lemma~\ref{singer1}. Therefore, $\mathcal{L}^{p,q}(b_\sigma)=0$ for every $\sigma\in\pgal(\mathcal{P}/F)$, where $b_\sigma$ is defined as in the discussion preceding Lemma~\ref{singer1}. We need to rule out the possibility that there exist $\smash{\tilde{\mathcal{L}}}\in K_0[\partial_t]$ and $\smash{\tilde{h}}\in K$ satisfying part \eqref{tcond} of Lemma~\ref{singer1} with \begin{equation*}\smash{\tilde{\ell}}:=\mathrm{ord}(\smash{\tilde{\mathcal{L}}})<\mathrm{ord}(\mathcal{L}^{p,q})=\ell.\end{equation*} 

We proceed by contradiction: assume such $\smash{\tilde{\mathcal{L}}}$ and $\smash{\tilde{h}}$ do exist. Let $\smash{\tilde{\mathcal{L}}}:=\sum_{i=0}^{\smash{\tilde{\ell}}}\smash{\tilde{a}}_i\partial_t^i$, with $\smash{\tilde{a}}_i\in K_0$. By Lemma~\ref{singer1}, $\smash{\tilde{h}}$ satisfies the inhomogeneous equation \vspace{-.1in}\begin{equation}\label{bad3}\smash{\delta_x\tilde{h}+p\tilde{h}=\sum_{i=0}^{\tilde{\ell}}\tilde{a}_iR_i}. \vphantom{\sum_i}\end{equation} If the bounds $S_{\smash{\tilde{\ell}}}$ and $T_{\smash{\tilde{\ell}}}$ of Step~\ref{step2.4} exist, then $\smash{\tilde{h}}$ must have the following form  \begin{equation}\label{bad1} \smash{\tilde{h}}=\sum_{j=0}^{S_{\smash{\tilde{\ell}}}}\smash{\tilde{b}}_jx^j+\sum_{k=1}^{T_{\smash{\tilde{\ell}}}}\left(\frac{\sum_{l=0}^m\smash{\tilde{c}}_{k,l}x^l}{d^k}\right). \end{equation} This implies that setting $\smash{\tilde{a}}_i=\alpha_i$, $\smash{\tilde{b}}_j=\beta_j$, and $\smash{\tilde{c}}_{k,l}=\xi_{k,l}$ for each $i,j,k,l$ is a solution for the system of homogeneous linear equations defined by~\eqref{ihomh} with $N=\smash{\tilde{\ell}}$. This contradicts the choice of~$\ell$.

It remains to be shown that the bounds $S_{\smash{\tilde{\ell}}}$ and $T_{\smash{\tilde{\ell}}}$ of Step~\ref{step2.4} exist. We will compute these bounds directly, using a modification of the argument given in \cite[Lemma 3.1]{singer:1981}, and we will see that they depend only on $p$, $q$, and $\smash{\tilde{\ell}}$. Let $e\in K_0$ be a pole of $\smash{\tilde{h}}\in K_0(x)$. We claim that $e$ must be a pole of either $p$ or $q$. To see this, suppose on the contrary that $e$ is a pole of $\smash{\tilde{h}}$, but not of $p$ or $q$. If $-T$ is the order of $\smash{\tilde{h}}$ at $e$, with $T\geq 1$, then the order of $\delta_x\smash{\tilde{h}}+p\smash{\tilde{h}}$ is $-T-1$ (because $e$ is not a pole of $p$), while the order of $\sum_{i=0}^{\smash{\tilde{\ell}}}\smash{\tilde{a}}_iR_i$ is nonnegative (because $e$ is not a pole of $q$); this gives a contradiction. Now let $e$ be a pole of $p$ or $q$, and let \begin{equation*} \sum_{i=0}^{\smash{\tilde{\ell}}}\smash{\tilde{a}}_iR_i=\frac{w_1}{(x-e)^{\tau_1}}+\text{(higher order terms)} \qquad\quad\begin{aligned} p&=\frac{w_2}{(x-e)^{\tau_2}}+\text{(higher order terms)} \\ \smash{\tilde{h}}&=\frac{w_3}{(x-e)^T}+\text{(higher order terms)}.\end{aligned} \end{equation*} be the $(x-e)$-adic expansions of these elements, with $w_2\in \overline{C(t)}^\times$, and $w_1,w_3\in K_0^\times$. If we substitute these expressions in \eqref{bad3}, we obtain \begin{equation}\label{poles}\!\!\! \biggl(\frac{-Tw_3}{(x-e)^{T+1}}+\text{(higher order terms)}\biggr)+ \left(\frac{w_2w_3}{(x-e)^{T+\tau_2}}+\text{(higher order terms)}\right)  =\frac{w_1}{(x-e)^{\tau_1}}+\text{(higher order terms)}. \end{equation}We see that either $\mathrm{max}\{T+1,T+\tau_2\}=\tau_1$, or else the lowest order terms in the left-hand side of \eqref{poles} must cancel. In order for this cancellation to take place, it is necessary that $T+1=\tau_2+1$ and $Tw_3=w_2w_3$, or, equivalently, $\tau_2=1$ and $T=w_2$. In any case, we have that $T\leq\mathrm{max}\{w_2,\tau_1-1\}$ (if $w_2$ is not an integer, we deduce that $T\leq\tau_1-1$). Evidently, $w_2$ depends only on $p$. Although $\tau_1$ depends on the $\smash{\tilde{a}}_i$, as well as on $q$ and $\smash{\tilde{\ell}}$, one can easily find a (generic) upper bound for $\tau_1$ which depends only on $q$ and $\smash{\tilde{\ell}}$, and not on the specific values of the $\smash{\tilde{a}}_i$; namely, take $\tau_1:=\mathrm{min}_i\{\mathrm{ord}_e(R_i) \ | \ 0\leq i \leq \smash{\tilde{\ell}}\}$. Now we may take $T_{\smash{\tilde{\ell}}}$ to be greater than any of the finitely many bounds so obtained for each pole of $p$ and $q$. 

The bound $S_{\smash{\tilde{\ell}}}$ is computed as follows: let $\smash{\tilde{S}}$ be an upper bound for the degree of the polynomial part of each $R_i$, for $0\leq i \leq \smash{\tilde{\ell}}$. It follows from \eqref{bad3} that $\mathrm{max}\{ S_{\smash{\tilde{\ell}}}-1,\mathrm{deg}_x(p_0)+S_{\smash{\tilde{\ell}}} \}\leq \smash{\tilde{S}}$, where $p_0$ is defined as in Step~\ref{step2.1}. This concludes the proof that the bounds $T_{\smash{\tilde{\ell}}}$ and $S_{\smash{\tilde{\ell}}}$ of Step~\ref{step2.4} exist. \end{proof}

\begin{rem}\label{compute-bounds}
In the preceding proof, we showed how to compute the bounds $S_N$ and $T_N$ of Step~\ref{step2.4} by working with the $(x-e)$-adic expansions of $p$ and $q$, for each linear factor $(x-e)$ of $d$ over $\overline{C(t)}[x]$. In practice, it may not be feasible to compute a factorization of $d$ into irreducible factors over $C(t)$, let alone linear factors over $\overline{C(t)}$. We will now indicate how to obtain the bounds $S_N$ and $T_N$ \emph{without} factorizing $d$.

Let $(a_0,\dots,a_N)$ be any $(N+1)$-tuple of elements of $K_0$, and let $h\in K$ be a solution for the equation obtained from \eqref{ihomy} by replacing $\alpha_i$ with $a_i$ for each $i$:\vspace{-.133in}\begin{equation}\label{bad2} \delta_xY+pY=\sum_{i=0}^Na_iR_i.\end{equation}If $(x-e)$ is a factor of $d$ in $\overline{C(t)}[x]$, the argument given in the proof of Proposition~\ref{alg2} shows that $h$ can have a pole at $e$ of order at most $\tau_1-1$, where $\tau_1$ is defined as in the proof of Proposition~\ref{alg2}, \emph{unless} $p$ has a pole at $e$ of order exactly $1$. Given this a priori bound, it suffices to bound the order of a pole of $h$ at the simple poles of $p$. 

We now compute the \emph{divisor of simple poles} of $p$: \begin{equation*}  d_1:=\frac{\mathrm{gcd}_x(d,p_2)}{\mathrm{gcd}_x(d,p_2,\delta_xp_2)}. \end{equation*} This is the product of all linear factors of $p_2$ which have multiplicity one in $p_2$. We will mimic the argument given at the end of Proposition~\ref{alg2}, using powers of $d_1$ as denominators, instead of powers of $(x-e)$. Write \begin{equation*} \sum_{i=0}^Na_iR_i  =\frac{w_1}{d_1^{\tau}}+\text{(irrelevant terms)}, \quad p  =\frac{w_2}{d_1}+\text{(irrelevant terms)} ,\quad\text{and}\quad h  = \frac{w_3}{d_1^T}+\text{(irrelevant terms)}, \end{equation*} where $ w_2\in \overline{C(t)}[x]$, $w_1,w_3\in K_0[x]$, and $\mathrm{deg}_x(w_j)<\mathrm{deg}_x(d_1)$ for each $1\leq j \leq 3$. A term in $\sum_ia_1R_i$, $p$, or $h$ is deemed ``irrelevant'' if its order at \emph{every} linear factor of $d_1$ is higher than $-\tau $, $-1$, or $-T$, respectively. In order to bound $T$, we proceed as before: substituting these expressions in \eqref{bad2}, we obtain \begin{equation}\label{poles2} \left(\frac{-Tw_3\delta_xd_1}{d_1^{T+1}}+\text{(irrelevant terms)}\right)+\left(\frac{w_2w_3}{d_1^{T+1}}+\text{(irrelevant terms)}\right) = \frac{w_1}{d_1^{\tau }}+\text{(irrelevant terms)}.\end{equation} It follows that $T+1\geq\tau $. If this inequality is strict, then the lowest order terms in the left-hand side of \eqref{poles2} must cancel. For such cancellation to take place, it is necessary for $(w_2w_3-Tw_3\delta_xd_1)$ to be divisible by $d_1$ in $K_0[x]$. This is almost what we want, except that this condition on $T$ depends on $w_3$, which in turn depends on the $a_i$. However, since $\mathrm{deg}_x(w_3)<\mathrm{deg}_x(d_1)$, we have the implication \begin{equation*}  d_1\big|\big(w_3(w_2-T\delta_xd_1)\big)\Longrightarrow\mathrm{gcd}_x\big(d_1, (w_2-T\delta_xd_1)\big)\neq 1.\end{equation*} If we now perform the Euclidean algorithm to compute this $\mathrm{gcd}_x$, we will obtain a series of residues. These residues will be polynomials in $x$, whose coefficients are polynomials in $\overline{C(t)}[T]$. At least one of these residues must be zero in order for the $\mathrm{gcd}_x$ to be different from $1$. Therefore, $T$ must satisfy at least one of a finite number of explicitly constructed polynomial equations with coefficients in $\overline{C(t)}$, and we may obtain a bound on the size of any integer solution for a given polynomial equation, depending only on the coefficients of the given polynomial. Having thus obtained the bound $T_N$, we obtain the bound $S_N$ by comparing the degrees of the polynomial parts of the right- and left-hand sides of \eqref{bad2}, after replacing $Y$ with $h$ and clearing denominators.\end{rem}


\subsection{The $\mathrm{PPV}$ group of an intersection of $\mathrm{PPV}$ fields}\label{comparison-sec} In this section we give an algorithm to compute the $\mathrm{PPV}$ group corresponding to the intersection of two $\mathrm{PPV}$ fields of the form considered at the beginning of this section. We will apply this algorithm in \S\ref{recover} to recover the $\mathrm{PPV}$ group corresponding to a second order equation from a pair of auxiliary differential algebraic groups.

Let $\eta_1,\eta_2\in\overline{C(t)}(x)$, and consider the first order inhomogeneous equation $\delta_xY=\eta_ r $, for $ r =1,2$. If we let $p_ r :=\frac{\delta_x\eta_ r }{\eta_ r }$, the discussion at the beginning of this section shows that the corresponding homogeneous equations are given by the operators \vspace{-.125in}\begin{equation*}\mathcal{P}_1:=\delta_x^2-p_1\delta_x\quad\text{and}\quad\mathcal{P}_2:=\delta_x^2-p_2\delta_x,\end{equation*} and that $\pgal(\mathcal{P}_ r /K)$ is a differential algebraic subgroup of $\mathbb{G}_a(K_0)$. Let $E$ denote the $\mathrm{PPV}$ extension of $\mathrm{PPV}(\mathcal{P}_1/K)$ corresponding to the operator $\mathcal{P}_2$, and let $\theta_1,\theta_2\in E$ such that $\delta_x\theta_ r =\eta_ r $. Then, we have that \begin{equation*} K\langle\theta_ r \rangle_\Delta=\mathrm{PPV}(\mathcal{P}_ r /K)\subseteq E=K\langle\theta_1,\theta_2\rangle_\Delta, \ \text{and we may define}\ L:=\mathrm{PPV}(\mathcal{P}_1/K)\cap\mathrm{PPV}(\mathcal{P}_2/K)\subseteq E.  \end{equation*} The group $\pgal(\mathcal{P}_ r /L)$ is normal in $\pgal(\mathcal{P}_ r /K)$, since the latter is abelian. It follows from the Galois correspondence (Theorem~\ref{pgal}) that $L$ is a $\mathrm{PPV}$ extension of $K$ (for some operator). If we denote the corresponding $\mathrm{PPV}$ group by $\Lambda$, the Galois correspondence implies that we have surjections $\pi_ r :\pgal(\mathcal{P}_ r /K)\twoheadrightarrow\Lambda$, and that \begin{equation*}\pgal(E/K)\simeq\pgal(\mathcal{P}_1/K)\times_\Lambda\pgal(\mathcal{P}_2/K),\end{equation*} where the fibered product\footnote{Recall that this is the subgroup of elements $(\sigma_1,\sigma_2)\in\pgal(\mathcal{P}_1/K)\times\pgal(\mathcal{P}_2/K)$ such that $\pi_1(\sigma_1)=\pi_2(\sigma_2)$.} is taken with respect to the differential-algebraic homomorphisms $\pi_ r $.

The following algorithm computes $\Lambda$, as well as the maps $\pi_1$ and $\pi_2$.


\begin{algo}[(Intersection)]\label{comparison}$\hphantom{X}$
\begin{center}
\begin{boxedminipage}{14cm}\textbf{Input:} $\eta_1,\eta_2\in\overline{C(t)}(x)$.

\hypertarget{output5}{\textbf{Output:}} A differential algebraic group $\Lambda^{\eta_1,\eta_2}$, and differential algebraic homomorphisms $\pi_ r ^{\eta_1,\eta_2}:\pgal(\mathcal{P}_ r )\twoheadrightarrow\Lambda^{\eta_1,\eta_2}$, for $ r =1,2$, defined by \begin{equation*}\Lambda^{\eta_1,\eta_2}:=\mathbb{G}_a(K_0;\mathcal{L}'')\quad\text{and}\quad \pi_ r ^{\eta_1,\eta_2}:\sigma\mapsto\mathcal{L}_ r '(\sigma),\end{equation*} such that the $\mathrm{PPV}$ group of $\mathrm{PPV}(\mathcal{P}_1/K)\cap\mathrm{PPV}(\mathcal{P}_2/K)$ is isomorphic to $\Lambda^{\eta_1,\eta_2}$ and $\pi_ r ^{\eta_1,\eta_2}$ is the corresponding surjection of $\mathrm{PPV}$ groups (cf. Theorem \ref{pgal}).\end{boxedminipage}
\end{center}

\begin{asparaenum}[\bfseries {Step} 1:]

\item \label{step5.1} Apply Algorithm~\ref{rational} with inputs $\eta_1$ and $\eta_2$, and let $\mathcal{L}_1:=\mathcal{L}^{\eta_1}$ and $\mathcal{L}_2:=\mathcal{L}^{\eta_2}$ be the \hyperlink{output1}{Outputs}. If $\mathrm{ord}(\mathcal{L}_1)\geq\mathrm{ord}(\mathcal{L}_2)$, proceed to Step~\ref{step5.2}. 

Otherwise, reverse the roles of $\eta_1$ and $\eta_2$, so that $\mathrm{ord}(\mathcal{L}_1)\geq\mathrm{ord}(\mathcal{L}_2)$, and then proceed to Step~\ref{step5.2}.

\item \label{step5.2} Write \vspace{-.115in}\begin{equation*}\smash{ \eta_1=\eta_1^{(0)}+\frac{\eta_1^{(1)}}{\eta_1^{(2)}}\quad\text{and}\quad \eta_2=\eta_2^{(0)}+\frac{\eta_2^{(1)}}{\eta_2^{(2)}}},\vphantom{\sum_i}  \end{equation*} with\footnote{The parenthetical superscript in $\eta_j^{(i)}$ is an index, and does not denote an $i^{\text{th}}$ derivative.} $\eta_1^{(i)},\eta_2^{(i)}\in\overline{C(t)}[x]$ for $0\leq i\leq 2$, such that \begin{gather*} \smash{\mathrm{gcd}_x(\eta_1^{(1)},\eta_1^{(2)})=1=\mathrm{gcd}_x(\eta_2^{(1)},\eta_2^{(2)})}, \\ \mathrm{deg}_x(\eta_1^{(1)})<\mathrm{deg}_x(\eta_1^{(2)})\quad\text{and}\quad\mathrm{deg}_x(\eta_2^{(1)})<\mathrm{deg}_x(\eta_2^{(2)}). \end{gather*}

\item \label{step5.3} Let \vspace{-.123in}\begin{equation*} \smash{d:=\frac{\eta_1^{(2)}\eta_2^{(2)}}{\mathrm{gcd}_x\Bigl(\eta_1^{(2)}\eta_2^{(2)},\delta_x\bigl(\eta_1^{(2)}\eta_2^{(2)}\bigr)\Bigr)}},\vphantom{\sum_{\substack{i \\ i}}}  \end{equation*} an let $n\in\mathbb{N}$ be the smallest nonnegative integer such that $d^{n}\eta_ r $ is a polynomial in $x$, for each $ r =1,2$. We remark that $d$ is the product of the irreducible factors of $\eta_1^{(2)}$ and $\eta_2^{(2)}$, and $n$ is the highest multiplicity of an irreducible factor of $\eta_1^{(2)}$ or $\eta_2^{(2)}$.

\item \label{step5.4} Set \begin{equation*}\ell_ r   :=\mathrm{ord}(\mathcal{L}_ r ),\qquad  \nu  :=\ell_1-\ell_2,\qquad s  :=\mathrm{max}\left\{\mathrm{deg}_x(\eta_1^{(0)}),\mathrm{deg}_x(\eta_2^{(0)})\right\}, \qquad\text{and}\qquad M  :=\mathrm{deg}_x(d).\end{equation*} For each $\nu\leq N \leq\ell_1$, write \begin{equation*} \mathcal{L}_N^{(1)}  := \sum_{i=0}^N \alpha_i\partial_t^i, \qquad \mathcal{L}_N^{(2)}  := \sum_{j=0}^{N-\nu}\beta_j\partial_t^j,\qquad\text{and}\qquad f_N  := \sum_{k=0}^{s+1}\gamma_kx^k +\sum_{l=1}^{(n+N-1)}\left( \frac{\sum_{m=0}^{M-1}\xi_{l,m}x^m}{d^l} \right),  \end{equation*} where the $\alpha_i$, $\beta_j$, $\gamma_k$, and $\xi_{l,m}$ are undetermined coefficients.

\item \label{step5.5} Treating each $\gamma_k$ and $\xi_{l,m}$ as a $\delta_x$-constant, set \begin{equation*}\label{ihom-comp}\tag{$\mathbf{J}_N$} \mathcal{L}_N^{(1)}(\eta_1)-\mathcal{L}_N^{(2)}(\eta_2)=\delta_xf_N,\end{equation*} and then multiply each side of this equation by $d^{n+N}$. The result will be an equality of polynomials in $x$, whose coefficients are linear forms in the $\alpha_i$, $\beta_j$, $\gamma_k$, and $\xi_{l,m}$. Equating coefficients of like-powers of $x$, we obtain a system of\vspace{-.125in} \begin{equation*}  M(n+N)+s \end{equation*} homogeneous linear equations with coefficients in $\overline{C(t)}$, in the \begin{equation*} (N+1)+(N-\nu+1)+(s+2)+M(n+N-1)  \vspace{-.1in}\end{equation*} variables $\alpha_i$, $\beta_j$, $\gamma_k$, $\xi_{l,m}$.

\item \label{step5.6} If $N=\ell_1$, the system of linear equations defined by~\eqref{ihom-comp} has a solution with not all $\alpha_i$ being zero and not all $\beta_j$ being zero. Find the smallest nonnegative integer $\omega\leq\ell_1$ such that the system of linear equations defined by~\eqref{ihom-comp}, with $N=\omega$, has a solution with not all $\alpha_i$ being zero.

\item \label{step5.7} If $\omega=\ell_1$, set $\mathcal{L}_1':=\mathcal{L}_1$, $\mathcal{L}_2':=\mathcal{L}_2$, and $\mathcal{L}'':=1$, and then go to \hyperlink{output5}{Output}.

Otherwise, find a solution \begin{align*}  \alpha_i&=a_i \quad &&\text{for}\ 0\leq i \leq \omega, \\ \beta_j&=b_j\quad &&\text{for}\ 0\leq j \leq \omega-\nu,\\ \gamma_k & = c_k \quad && \text{for}\  0\leq k\leq s+1\\ \xi_{l,m}&=e_{l,m} \quad &&\text{for}\ 1\leq l \leq n+N-1\ \text{and}\  0 \leq m \leq M-1;\end{align*} such that each $a_i,b_j,c_k,e_{l,m}\in\overline{C(t)}$ and $a_\omega=1$. This condition determines $a_0,\dots,a_\omega$ and $b_0,\dots,b_{\omega-\nu}$ uniquely.

\item \label{step5.8} Set $\mathcal{L}_1':=\sum_{i=0}^\omega a_i\partial_t^i$ and $\mathcal{L}_2':=\sum_{j=0}^{\omega-\nu}b_j\partial_t^j$, let $\mathcal{L}''\in\overline{C(t)}[\partial_t]$ be the unique monic operator such that\footnote{The ring $\overline{C(t)}[\partial_t]$ is a left and right Euclidean domain, and there are algorithms to compute this $\mathcal{L}''$ (see \cite[\S2.1]{vanderput-singer:2003}).} $\mathcal{L}_1=\mathcal{L}''\mathcal{L}_1'$, and go to \hyperlink{output5}{Output}.
\end{asparaenum}
\end{algo} 

\begin{eg} We will now apply Algorithm \ref{comparison} with inputs $\eta_1:=\frac{x^2+t^2x+t}{x^3+tx}$ and $\eta_2:=\frac{2t}{x^2+t}$; note that $\eta_1=\frac{1}{2}t\eta_2+\frac{1}{x}$. The first step is to apply Algorithm \ref{rational} with inputs $\eta_1$ and $\eta_2$, to obtain the operators $\mathcal{L}_1:=\frac{\partial^2}{\partial t^2}-\frac{1}{2t}\frac{\partial}{\partial t}$ and $\mathcal{L}_2:=\frac{\partial}{\partial t}-\frac{1}{2t}$ (cf. Example~\ref{example1}). The polynomial $d$ of Step \ref{step5.3} is given by $d=x^3+tx$. We then compute the integers $n=1$, $\ell_1=2$, $\ell_2=1$, $\nu=1$, $M=3$, $s=0$, and write down the pair of operators and the rational function with undetermined coefficients of Step~\ref{step5.4}, with $N=1$: \begin{equation*} \smash{\mathcal{L}_1^{(1)} :=\alpha_1\partial_t+\alpha_0, \qquad \mathcal{L}_1^{(2)}  :=\beta_0,  \qquad \text{and}\qquad \smash{f_1:=\gamma_0+\gamma_1x+\frac{\xi_{1,0}+\xi_{1,1}x+\xi_{1,2}x^2}{x^3+tx}}.}\end{equation*} Then, we substitute these expressions in equation \eqref{ihom-comp} as in Step~\ref{step5.5}, to obtain \begin{gather} \label{eg-comp-eq} \alpha_1\tfrac{\partial}{\partial t}\left( \tfrac{x^2+t^2x+t}{x^3+tx}\right)+\alpha_0\cdot\tfrac{x^2+t^2x+t}{x^3+tx}-\beta_0\cdot\tfrac{2t}{x^2+t}=\tfrac{\partial}{\partial x}\left(\gamma_0+\gamma_1x+\tfrac{\xi_{1,0}+\xi_{1,1}x+\xi_{1,2}x^2}{x^3+tx}  \right). \\ \intertext{After expanding this out and doing some simplification, we obtain} \notag  \tfrac{-\alpha_1t^2x^2}{(x^3+tx)^2} + \tfrac{\alpha_0x^2+(\alpha_0t^2+2\alpha_1t)x+\alpha_0t}{x^3+tx}-\tfrac{2t\beta_0}{x^2+t}=\gamma_1+\tfrac{\xi_{1,1}+2\xi_{1,2}x}{x^3+tx}-\tfrac{(\xi_{1,0}+\xi_{1,1}x+\xi_{1,2}x^2)(3x^2+t)}{(x^3+tx)^2}  .\end{gather} After multiplying by $d^2=(x^3+tx)^2$ on both sides, and then subtracting the left-hand side from the right-hand side, we obtain that \eqref{eg-comp-eq} holds if and only if the following polynomial in $x$ is zero: \begin{multline*}  \gamma_1  x^6 -\alpha_0x^5 +\bigl( -t^2\alpha_0-2t\alpha_1-2t\beta_0+2t\gamma_1-\xi_{1,2}  \bigr)  x^4 +\bigl( -2t\alpha_0-2\xi_{1,1} \bigr) x^3  \\ +\bigl(-t^3\alpha_0-t^2\alpha_1-2t^2\beta_0+t^2\gamma_1-3\xi_{1,0}+t\xi_{1,2}     \bigr)   x^2  -t^2\alpha_0 x -t  \xi_{1,0}x^0. \end{multline*} Setting each coefficient equal to zero yields the system of homogeneous linear equations defined in Step~\ref{step5.5}. We verify that $\alpha_0=0$, $\alpha_1=1$, $\beta_0=\frac{3}{4}$, $\gamma_1=0$, $\xi_{1,0}=0$, $\xi_{1,1}=0$, $\xi_{1,2}=-\frac{1}{2}t$ is a solution, and that indeed $\smash{\frac{\partial}{\partial t}\eta_1-\frac{3}{4}\eta_2=\frac{\partial}{\partial x}\bigl( \frac{-tx^2}{2(x^3+tx)} \bigr)}$. The smallest value for $N$ allowed by Step~\ref{step5.4} is $N=1$, since we assume $N\geq\nu=1$. Therefore, $\omega=1$, and we have $\smash{\mathcal{L}_1'=\frac{\partial}{\partial t}}$, $\smash{\mathcal{L}_2'=\frac{3}{4}}$, and $\smash{\mathcal{L}''=\frac{\partial}{\partial t}-\frac{1}{2t}}$. \end{eg}

Our proof that Algorithm~\ref{comparison} gives the right answer will rely on the following well-known result.

\begin{thm}[(Kolchin-Ostrowski \cite{kolchin:1968})]\label{kolostro} Let $E$ be a $\delta_x$-field extension of $K$ such that $E^{\delta_x}=K_0$, and let $\mathfrak{e}_1,\dots,\mathfrak{e}_m,\mathfrak{f}_1,\dots,\mathfrak{f}_n\in E$ such that $\tfrac{\delta_x\mathfrak{e}_i}{\mathfrak{e_i}}\in K$ for each $1\leq i\leq m$ and $\delta_x\mathfrak{f}_j\in K$ for each $1\leq j \leq n$.

Then, there exists a nonzero polynomial $G\in K[X_1,\dots,X_m,Y_1,\dots,Y_n]$ such that $G(\mathfrak{e}_1,\dots ,\mathfrak{e}_m,\mathfrak{f}_1,\dots,\mathfrak{f}_n)=0$ if and only if at least one of the following holds:
\begin{enumerate}
\item There exist integers $m_i\in\mathbb{Z}$, not all zero, such that $\prod_{i=1}^m\mathfrak{e}_i^{m_i}\in K.$

\item There exist elements $c_j\in K^{\delta_x}$, not all zero, such that $\sum_{j=1}^nc_j\mathfrak{f}_j\in K.$

\end{enumerate}
\end{thm}

\begin{prop}[(Algorithm~\ref*{comparison} is correct)] \label{comparison-proof} Let $\eta_1,\eta_2\in\overline{C(t)}(x)$, and let $\Lambda^{\eta_1,\eta_2}$ and $\pi_ r ^{\eta_1,\eta_2}$ be the \hyperlink{output5}{Outputs} of Algorithm~\ref{comparison}. Define the operators $\mathcal{P}_ r \in K[\delta_x]$ for $ r =1,2$, as well as the fields $E$ and $L$ as above, that is\begin{equation*}E  :=\mathrm{PPV}\bigl(\mathcal{P}_2/\mathrm{PPV}(\mathcal{P}_1/K)\bigr),\quad \text{and} \quad L  :=\mathrm{PPV}(\mathcal{P}_1/K)\cap\mathrm{PPV}(\mathcal{P}_2/K)\subseteq E.\end{equation*} Then, $\pi_ r ^{\eta_1,\eta_2}:\pgal(\mathcal{P}_ r /K)\twoheadrightarrow\Lambda^{\eta_1,\eta_2}$ is a surjection, and $L$ is the $\mathrm{PPV}$ extension corresponding to $\mathrm{ker}(\pi_ r ^{\eta_1,\eta_2})$ by Theorem~\ref{pgal} applied to $\mathrm{PPV}(\mathcal{P}_ r /K)$, for each $ r =1,2$.
\end{prop}

\begin{proof} Let $\theta_ r \in E$ such that $\delta_x\theta_ r =\eta_ r $. Then, $\mathrm{PPV}(\mathcal{P}_ r /K)=K\langle\theta_ r \rangle_\Delta$, and since $\delta_x\partial_t^n\theta_ r =\partial_t^n\eta_ r \in K$, it is enough to take derivatives with respect to $\partial_t$ only; that is, $\mathrm{PPV}(\mathcal{P}_ r /K)=K\langle\theta_ r \rangle_{\partial_t}$. If we define $\ell_ r $ as in Step~\ref{step5.4}, Lemma~\ref{singer1} implies that \vspace{-.125in} \begin{equation*} \mathrm{PPV}(\mathcal{P}_ r /K)=K\langle\theta_ r \rangle_{\partial_t}=K(\theta_ r ,\partial_t\theta_ r ,\dots,\partial_t^{\ell_ r -1}\theta_ r ),\end{equation*} and the transcendence degree of $\mathrm{PPV}(\mathcal{P}_r/K)$ over $K$ is $\ell_r$. 

The map $\sigma\mapsto(\sigma\theta_r-\theta_r)$ identifies $\pgal(\mathcal{P}_r/K)$ with $\mathbb{G}_a(K_0;\mathcal{L}_r)$, for each $r=1,2$. By Theorem~\ref{pgal}, there is a unique $\partial_t$-algebraic subgroup $H_r\leq\pgal(\mathcal{P}_r/K)$ such that $L$ is the fixed field of $H_r$; that is, $L$ is precisely the set of elements $a\in\mathrm{PPV}(\mathcal{P}_r/K)$ such that $\sigma a =a$ for every $\sigma\in H_r$. On the other hand, if $\mathcal{L},\mathcal{L}'\in K_0[\partial_t]$ are such that $\mathbb{G}_a(K_0;\mathcal{L}')\subseteq\mathbb{G}_a(K_0;\mathcal{L})$, then there exists $\mathcal{L}''\in K_0[\partial_t]$ such that $\mathcal{L}=\mathcal{L}''\mathcal{L}'$, and $\mathbb{G}_a(K_0;\mathcal{L})/\mathbb{G}_a(K_0;\mathcal{L}')\simeq\mathbb{G}_a(K_0;\mathcal{L}'')$ (see \cite[p.45]{vanderput-singer:2003}). If $\mathcal{L}_r'\in K_0[\partial_t]$ is a right-hand factor of $\mathcal{L}_r$, the fixed field of $\mathbb{G}_a(K_0;\mathcal{L}_r')$ is precisely $K\langle\mathcal{L}_r'(\theta_r)\rangle_\Delta$. Therefore, there is a right-hand factor $\mathcal{L}_r'$ of $\mathcal{L}_r$ such that $L=K\langle \mathcal{L}_r'(\theta_r)\rangle_\Delta$.

We remark that every element $a\in L\backslash K$ defines an algebraic relation, as follows: since $a\in\mathrm{PPV}(\mathcal{P}_1/K)$, there is a rational function $\Phi\in K(Y_1,\dots ,Y_{\ell_1})$ such that $a=\Phi(\theta_1,\dots ,\partial_t^{\ell_1-1}\theta_t)$. Similarly, there is a rational function $\Psi\in K(Y_1,\dots ,Y_{\ell_2})$ such that $a=\Psi(\theta_2,\dots , \partial_t^{\ell_2-1}\theta_2)$. Setting these two expressions for $a$ equal to each other and clearing denominators, we obtain that $L\neq K$ if and only if there is a nontrivial $K$-algebraic relation amongst the elements  $\theta_1,\theta_2,\partial_t\theta_1,\partial_t\theta_2,\dots , \partial_t^{\ell_1-1}\theta_1, \partial_t^{\ell_2-1}\theta_2\in E$.

By Definition~\ref{ppv-def}, we have that $E^{\delta_x}=K_0=K^{\delta_x}$, and we may apply Theorem~\ref{kolostro} with \begin{equation*} \mathfrak{f}_j:=\begin{dcases}\partial_t^j\theta_1 & \text{for}\  0\leq j\leq \ell_1-1, \ \text{and} \\ \partial_t^{j-\ell_1}\theta_2 & \text{for}\  \ell_1\leq j\leq \ell_1+\ell_2-1.  \end{dcases}\end{equation*} Therefore, if $L\neq K$, we have that there exist elements $\tilde{a}_0,\dots \tilde{a}_{\ell_1-1}, \tilde{b}_0,\dots , \tilde{b}_{\ell_2-1}\in K_0$, not all zero, such that \begin{equation}\label{kolostro-2}\sum_{i=0}^{\ell_1-1}\tilde{a}_i\partial_t^{i}\theta_1-\sum_{j=0}^{\ell_2-1}\tilde{b}_j\partial_t^j\theta_2\in K.\end{equation} There is a nonnegative integer ${\tilde{\omega}}\in\mathbb{N}$ which is minimal with respect to the property that~\eqref{kolostro-2} holds for some $\tilde{a}_i,\tilde{b}_j\in K_0$, with $\tilde{a}_i=0$ for all $i>{\tilde{\omega}}$, and $\tilde{a}_{\tilde{\omega}}\neq 0$, and if we impose the further condition that $\tilde{a}_{\tilde{\omega}}=1$, this determines $\tilde{a}_0,\dots,\tilde{a}_{\tilde{\omega}},\tilde{b}_0,\dots,\tilde{b}_{\ell_2-1}$ uniquely. The minimality of $\tilde{\omega}$ implies that $L$ is $\Delta$-generated over $K$ by $\sum_i\tilde{a}_i\partial_t^i\theta_1$ (and therefore also by $\sum_j\tilde{b}_j\partial_t^j\theta_2$).

We claim that $\omega=\tilde{\omega}$, where $\omega$ is defined as in Step~\ref{step5.6}. Note that this holds if and only if $a_i=\tilde{a}_i$, and $b_j=\tilde{b}_j$, where $a_i,b_j\in\overline{C(t)}$ are the elements found in Step~\ref{step5.7}. For definiteness, let us set \begin{equation}\label{kolostro-3}\sum_{i=0}^{{\tilde{\omega}}}\tilde{a}_i\partial_t^{i}\theta_1-\sum_{j=0}^{\ell_2-1}\tilde{b}_j\partial_t^{j}\theta_2=\tilde{f}\in K,\end{equation} where $a_{\tilde{\omega}}=1$. If we define $\tilde{\mathcal{L}}_1':=\sum_{i=0}^{\tilde{\omega}} \tilde{a}_i\partial_t^i$, $\tilde{\mathcal{L}}_2':=\sum_{j=0}^{\ell_2-1}\tilde{b}_j\partial_t^j$, and \begin{align*} \pi_ r :\pgal(\mathcal{P}_ r /K)  \longrightarrow  \ & \mathbb{G}_a(K_0) \\ \phantomarrow{XXXX}{\sigma} &\tilde{\mathcal{L}}_ r '(\sigma\theta_ r -\theta_ r )\end{align*} for $ r =1,2$, we have that $\pi_1\bigl(\pgal(\mathcal{P}_1/K)\bigr)=\pi_2\bigl(\pgal(\mathcal{P}_2/K)\bigr)$, as differential algebraic subgroups of $\mathbb{G}_a(K_0)$. If we let $\tilde{\mathcal{L}}''\in K_0[\partial_t]$ be the unique monic operator such that $\pi_1\bigl(\pgal(\mathcal{P}_1/K)\bigr)=\mathbb{G}_a(K_0;\tilde{\mathcal{L}}'')$, it follows that $\tilde{\mathcal{L}}''\tilde{\mathcal{L}}_ 1 '=\mathcal{L}_ 1$, and therefore $\ell_1-{\tilde{\omega}}=\mathrm{ord}(\tilde{\mathcal{L}}'')=\ell_2-\mathrm{ord}(\tilde{\mathcal{L}}_2')$, or in other words \begin{equation*}  \mathrm{ord}(\tilde{\mathcal{L}}_2')={\tilde{\omega}}-(\ell_1-\ell_2). \end{equation*} Therefore, $\tilde{b}_j=0$ for all $j>{\tilde{\omega}}-\nu$, and $\tilde{b}_{{\tilde{\omega}}-\nu}\neq 0$, where $\nu:=\ell_1-\ell_2$, as defined in Step~\ref{step5.4}. This implies that, if we apply $\delta_x$ on both sides of the equality in~\eqref{kolostro-3}, we obtain that $\tilde{a}_0,\dots,\tilde{a}_{\tilde{\omega}},\tilde{b}_0\dots,\tilde{b}_{{\tilde{\omega}}-\nu}$, together with the coefficients of a base-$d$ expansion of $\tilde{f}$, define a solution for the system of equations defined by~\eqref{ihom-comp}, with ${\tilde{\omega}}=N$. Therefore, $\omega={\tilde{\omega}}$ and we obtain our result. \end{proof}


\section{Algorithms for second order equations}\label{order2}
 
Let $K=K_0(x)$ be the $\Delta$-field defined at the beginning of \S\ref{order1}. In this section, we will describe an algorithm to compute the $\mathrm{PPV}$ group corresponding to a parameterized differential equation of the form \begin{equation}\label{original}\delta_x^2Y+r_1\delta_xY+r_2Y=0,\end{equation} with $r_i\in C(x,t)$ for $0\leq i\leq 2$. After performing a \emph{change of variables} (see \cite[Ex. 1.35 (5)]{vanderput-singer:2003} or \cite[p. 5]{kovacic:1986}), we may put this equation in the following form \vspace{-.1in} \begin{equation}\label{eq} \delta_x^2Y-rY=0,\vspace{-.05in}\end{equation} where $r:=\tfrac{1}{4}r_1^2+\tfrac{1}{2}\delta_xr_1-r_2$. 

The advantage of performing this change of variables is that, if $\mathcal{D}:=\delta_x^2-r$, then $\gal(\mathcal{D}/K)$ is an algebraic subgroup of $\mathrm{SL_2}(K_0)$ (and therefore $\pgal(\mathcal{D}/K)$ is a differential algebraic subgroup of $\mathrm{SL}_2(K_0)$), so that there are less possible candidates for (parameterized) Picard-Vessiot groups to consider. For the time being, we shall limit our attention to equations of the form~\eqref{eq}. In \S\ref{recover}, we will show how to express the $\mathrm{PPV}$ group of the original equation in terms of the $\mathrm{PPV}$ group of~\eqref{eq}.
 
 \subsection{Kovacic's algorithm}
 
 In \cite{kovacic:1986}, Kovacic describes an algorithm which: (i)~finds all Liouvillian solutions of~\eqref{eq}, whenever such solutions exist; and (ii)~computes $\gal(\mathcal{D}/K)$ from the data of these solutions (or their nonexistence). Kovacic's algorithm proceeds by first determining whether the Riccati equation \cite[\S4.1]{vanderput-singer:2003} for the operator $\mathcal{D}$, which is given by \vspace{-.125in}\begin{equation*}R_\mathcal{D}(U):=\delta_xU+U^2-r=0,\end{equation*} has a solution $u\in \bar{K}$. If such a $u$ exists, a computation shows that $\mathcal{D}$ factors over $K(u)[\delta_x]$ into first order factors, as $\mathcal{D}=(\delta_x+u)\circ(\delta_x-u)$ \cite[Ex. 1.36 (1b)]{vanderput-singer:2003}. Therefore, $\gal\smash{\bigl(\mathcal{D}/K(u)\bigr)}$ is reducible \cite[Prop. 2.1]{compoint-singer:1999}.
 
 \begin{defn}
 We denote the \emph{upper triangular} and \emph{infinite dihedral} subgroups of $\mathrm{SL_2}(K_0)$ by $\mathrm{UT}(K_0)$ and $\mathbf{D}_\infty(K_0)$, respectively. They are defined as \begin{equation*}\mathrm{UT}(K_0):=\left\{\begin{pmatrix}a & b \\0 & a^{-1}\end{pmatrix}\  \middle| \  a\in K_0^\times,b\in K_0\right\}\qquad\text{and}\qquad\ \mathbf{D}_\infty(K_0):=  \left\{\begin{pmatrix}  a & 0 \\ 0 & a^{-1} \end{pmatrix}, \begin{pmatrix} 0 & a \\ -a^{-1} & 0 \end{pmatrix}\  \middle|\  a\in K_0^\times\right\}.\end{equation*}
 \end{defn}
 
 For a proof of a significantly sharper version of the following result, see \cite[\S4.3.4]{vanderput-singer:2003} or Kovacic's original paper \cite[\S1.2]{kovacic:1986}. We have only included those parts which are relevant for our purposes.
 
 \begin{thm}[(Cases of Kovacic's algorithm)] \label{kov-cases}
 Kovacic's algorithm falls into four main cases:
 \begin{asparaenum}
 
 \item If $R_\mathcal{D}(u)=0$ for some $u\in K$, then $\gal(\mathcal{D}/K)$ is conjugate to a subgroup of $\mathrm{UT}(K_0)$.
 
 \item If $R_\mathcal{D}(u)=0$ for some $u\in F\backslash K$, with $F$ a quadratic extension of $K$, then $\gal(\mathcal{D}/K)$ is conjugate to a subgroup of $\mathbf{D}_\infty (K_0)$.
 
 \item If Cases (1) and (2) do not hold, and $R_\mathcal{D}(u)=0$ for some $u\in\bar{K}$, then $\gal(\mathcal{D}/K)$ is finite.
 
 \item If Cases (1), (2) and (3) do not hold, then $\gal(\mathcal{D}/K)=\mathrm{SL}_2(K_0)$.
 
 \end{asparaenum} 
 \end{thm} 

\begin{thm}[(Dreyfus {\cite[Prop. 9]{dreyfus:2011}})]
Suppose that $\gal(\mathcal{D}/K)=\mathrm{SL}_2(K_0)$. Then, $\pgal(\mathcal{D}/K)$ is conjugate to $\mathrm{SL}_2(C)$ if and only if the following inhomogeneous equation has a solution in $K$:\begin{equation}\label{dreyfus} \delta_x^3Y -4r\delta_xY-2(\delta_xr)Y=-2\partial_tr.\vspace{-.05in}\end{equation} 
\end{thm}

The following result is a corollary of Theorem \ref{zariski}.
 
\begin{thm} \label{finite-gps} $\gal(\mathcal{D}/K)$ is finite if and only if $\pgal(\mathcal{D}/K)$ is finite. In this case, $\pgal(\mathcal{D}/K)=\gal(\mathcal{D}/K)$.
\end{thm}

One can show that \eqref{dreyfus} has a solution in $K$ if and only if it has a solution in $\overline{C(t)}(x)$, and there is an algorithm to decide whether or not this holds \cite[Prop. 4.1 and Rmk. 4.5 (2)]{vanderput-singer:2003}. Therefore, Theorems~\ref{sl2}, \ref{dreyfus} and \ref{finite-gps} completely solve the problem of determining $\pgal(\mathcal{D}/K)$ if $\gal(\mathcal{D}/K)$ is finite or if ${\gal(\mathcal{D}/K)=\mathrm{SL}_2(K_0)}$.

We shall now restrict our attention to cases $(1)$ and $(2)$ of Kovacic's algorithm. From now on, we will assume that $\pgal(\mathcal{D}/K)$ is infinite, unless explicitly stated otherwise. Our strategy is similar to the one followed in \cite{dreyfus:2011}, and the arguments given there were of great help to us in developing ours. Since we do not have an interpretation for the solutions for our differential equations in terms of path integrals, we cannot directly apply these results in our context. We will give differential-algebraic versions of the arguments presented in \cite[\S4.1, \S4.2]{dreyfus:2011}, in order to prove that our Algorithms \ref{borel} and \ref{dihedral} are correct. In cases (1) and (2) of Kovacic's algorithm, we use the solutions for \eqref{eq} to define \emph{semi-invariants} of the linear differential algebraic group $\pgal(\mathcal{D}/K)$, in the sense of \cite[II,\S7]{cassidy:1972}. In the non-parameterized setting, this approach has led to generalizations and simplifications of Kovacic's algorithm which work over more general fields;\footnote{Kovacic's original algorithm assumes that the ground field $K$ is $\mathbb{C}(x)$.} see for example \cite{vanhoeij-weil:1997,ulmer-weil:2000}. See also \cite{vanhoeij:1997} for an efficient algorithm to compute first order right-hand factors of $\mathcal{D}$.
 

\subsection{Upper triangular case}
Suppose that $R_\mathcal{D}(U)=0$ has a solution $u\in K$. One can find such a $u\in \overline{C(t)}(x)$ explicitly (see \cite{ulmer-weil:2000,vanhoeij:1997} and \cite[Prop. 4.9]{vanderput-singer:2003}).


\begin{algo}[(Upper triangular)]\label{borel}$\hphantom{X}$

\begin{center}
\begin{boxedminipage}{10.7cm}\textbf{Input:} $u\in \overline{C(t)}(x)$.

\hypertarget{output3}{\textbf{Output:}} A differential algebraic group $\mathrm{UT}(K_0;A^u,B^u)$, defined by \begin{equation*}\mathrm{UT}(K_0;A^u,B^u):=\left\{\begin{pmatrix} a & b \\ 0 & a^{-1}\end{pmatrix}\ \middle|\ a\in A^u, b\in B^u\right\},\end{equation*} and such that $\pgal(\mathcal{D}/K)\simeq\mathrm{UT}(K_0;A^u,B^u)$.\end{boxedminipage}
\end{center}

\begin{asparaenum}[\bfseries {Step} 1:]

\item \label{step3.1}Apply Kovacic's algorithm to compute $\gal(\mathcal{D}/K)$. If it is found that \begin{equation*} \gal(\mathcal{D}/K)\simeq \left\{\begin{pmatrix} a & b \\ 0 & a^{-1} \end{pmatrix}\  \middle| \  a\in \mu_n, b\in K_0\right\} \end{equation*} for some $n\in\mathbb{N}$, set $A^u:=\mu_n$ (recall $\mu_n$ denotes the group of $n^\text{th}$ roots of unity in $C^\times$) and go to Step~\ref{step3.4}. 

Otherwise, go to Step~\ref{step3.2}. 

\item \label{step3.2} Apply Algorithm~\ref{rational} with input $\eta:=\partial_tu$, let $\mathcal{L}^\eta$ be its \hyperlink{output1}{Output}, and set \begin{equation*}A^u:=\mathbb{G}_m(K_0;\mathcal{L}^\eta).\end{equation*}

\item \label{step3.3} If $\mathcal{L}^\eta\neq 1$, let \begin{equation*}\smash{B^u:=\begin{cases} 0 & \text{if} \  \gal(\mathcal{D}/K)\simeq\mathbb{G}_m(K_0); \\ \mathbb{G}_a(K_0) & \text{otherwise},\end{cases}}\end{equation*} and go to \hyperlink{output3}{Output}.

If $\mathcal{L}^\eta=1$, proceed to Step~\ref{step3.4}. 

\item \label{step3.4}Find $v\in \overline{C(t)}(x)$ such that $\delta_xv=\partial_tu$. Such a $v$ exists whenever $A^u\subseteq\mathbb{G}_m(C)$.

\item \label{step3.5}Apply Algorithm~\ref{exponential} with inputs $p:=-2u$ and $q:=-2v$, and let $\mathcal{L}^{p,q}$ be its \hyperlink{output2}{Output}. Set \begin{equation*}B^u:=\mathbb{G}_a(K_0;\mathcal{L}^{p,q}),\vspace{-.07in}\end{equation*} and go to \hyperlink{output3}{Output}.

\end{asparaenum}
\end{algo}


\begin{prop}[(Algorithm~\ref{borel} is correct)] \label{borel-proof}Let $u\in \overline{C(t)}(x)$ such that $R_\mathcal{D}(u)=0$. Then, \begin{equation*}\pgal(\mathcal{D}/K)\simeq\mathrm{UT}(K_0;A^u,B^u).\end{equation*}\end{prop}

\begin{proof}[Proof (cf. {\cite[\S2.2]{dreyfus:2011}})]There is a basis $\{y_1,y_2\}$ of $\mathrm{Sol}(\mathcal{D})\subset\mathrm{PPV}(\mathcal{D}/K)$ such that $\delta_xy_1=uy_1$ and $\delta_x(\frac{y_2}{y_1})=y_1^{-2}$ \cite[\S1.3]{kovacic:1986}. This choice of basis identifies $\pgal(\mathcal{D}/K)$ with a subgroup of $\mathrm{UT}(K_0)$, as follows: for $\sigma\in\pgal(\mathcal{D}/K)$, define $a_\sigma:=\frac{\sigma y_1}{y_1}$, and $b_\sigma:=\frac{\sigma y_2}{y_1}-\frac{y_2}{\sigma y_1}$. A computation shows that $a_\sigma\in K_0^\times$,  $b_\sigma\in K_0$, and \begin{equation*}\sigma\mapsto\begin{pmatrix} a_\sigma & b_\sigma \\ 0 & a_\sigma^{-1}\end{pmatrix}\end{equation*} is the embedding $\pgal(\mathcal{D}/K)\hookrightarrow\mathrm{SL}_2(K_0)$ determined by the choice of basis $\{y_1,y_2\}$ of $\mathrm{Sol}(\mathcal{D})$.
 
Moreover, the map $\chi:\pgal(\mathcal{D}/K)\rightarrow\mathbb{G}_m(K_0):\sigma\mapsto a_\sigma$ is a differential-algebraic homomorphism (so its kernel is Kolchin-closed). The $\mathrm{PPV}$~extension corresponding to $\mathrm{ker}(\chi)$ by Theorem~\ref{pgal} is $F:=K\langle y_1\rangle_\Delta$, which is a $\mathrm{PPV}$~extension of $K$ for the operator $\delta_x-u$. The induced character $\chi:\pgal((\delta_x-u)/K)\rightarrow\mathbb{G}_m(K_0)$ is now injective, and its image is a differential algebraic subgroup $A$ of $\mathbb{G}_m(K_0)$. If $\sigma\in\mathrm{ker}(\chi)$, then $b_\sigma=\sigma(\frac{y_2}{y_1})-\frac{y_2}{y_1}$, so the map $\sigma\mapsto b_\sigma$ identifies $\pgal(\mathcal{D}/F)$ with a differential algebraic subgroup $B$ of $\mathbb{G}_a(K_0)$, and we have that \begin{equation*}\pgal(\mathcal{D}/K)\simeq \mathrm{UT}(K_0;A,B)=\left\{\begin{pmatrix} a & b \\ 0 & a^{-1}\end{pmatrix}\  \middle|\ a\in A, b\in B\right\}.\vspace{-.06in}\end{equation*} We will show that $A=A^u$ and $B=B^u$.
 
First we calculate $A$. If $\pgal(\mathcal{D}/K)$ is infinite, Theorem~\ref{zariski} implies that $A=\mu_n$ if and only if \begin{equation*} \gal(\mathcal{D}/K)\simeq \left\{\begin{pmatrix}  a & b \\ 0 & a^{-1}  \end{pmatrix} \ \middle|\ a\in \mu_n,\ b\in K_0\right\}.\end{equation*} If $A$ is infinite, a computation shows that $\sigma\big(\tfrac{\partial_ty_1}{y_1}\big)=\tfrac{\partial_ty_1}{y_1}+\tfrac{\partial_ta_\sigma}{a_\sigma}$ and $\delta_x\big(\tfrac{\partial_ty_1}{y_1}\big)=\partial_tu$. Applying Proposition~\ref{alg1} with $\eta=\partial_tu$ and $\theta=\tfrac{\partial_ty_1}{y_1}$, we deduce that $A=A^u$.
 
In order to calculate the operator defining $B$, we have to consider the possibilities $A\subseteq\mathbb{G}_m(C)$ and $A\nsubseteq\mathbb{G}_m(C)$ separately. The following result is proved in \cite{sit:1975} under different hypotheses, but the proof works in our setting as well.
 
 \begin{lem}\label{sit}
Let $B$ be a nontrivial proper differential algebraic subgroup of $\mathbb{G}_a(K_0)$, and let $a\in\mathbb{G}_m(K_0)$. Then, $B=aB:=\{ab\ | \ b\in B\}$ if and only if $\partial_ta=0$.
 \end{lem}
 
 \begin{proof}[Proof of Lemma]
 Let $\mathcal{L}\in K_0[\partial_t]$ be the unique monic operator defining $B$. We have that $\mathcal{L}\neq 0$ and $\mathrm{ord}(\mathcal{L})\geq 1$. If $a\in\mathbb{G}_m(K_0)$, the monic operator defining $aB$ is $\mathcal{L}':=a^{-1}\circ\mathcal{L}\circ a$. If $\partial_ta\neq 0$, then $\mathcal{L}\neq\mathcal{L}'$, so $B\neq aB$. \end{proof}
 
 \emph{Case 1:} $A\nsubseteq\mathbb{G}_m(C)$. The group $A$ acts on $B$ via \begin{equation*} \begin{pmatrix} a & 0 \\ 0 & a^{-1}  \end{pmatrix}\begin{pmatrix} 1& b \\  0 & 1 \end{pmatrix}\begin{pmatrix} a^{-1} & 0 \\ 0 & a  \end{pmatrix}=\begin{pmatrix} 1 & a^2b \\ 0 & 1  \end{pmatrix} .\end{equation*} By Lemma~\ref{sit}, in this case $B$ is either 0 or $\mathbb{G}_a(K_0)$, (because $a\in C\Leftrightarrow a^2\in C$). By Theorem~\ref{zariski}, $B=0$ if and only if $\gal(\mathcal{D}/K)\simeq\mathbb{G}_m(K_0)$, and we have that $B=B^u$ in this case.
 
 \emph{Case 2:} $A\subseteq\mathbb{G}_m(C)$. In this case, it is shown in \cite{cassidy-singer:2006} that there exists $v\in K$ such that $\partial_ty_1=vy_1$ and $\delta_xv=\partial_tu$. Note that this implies that $\delta_x\bigl(y_1^{-2}\bigr)=-2uy_1^{-2}$ and $\partial_t\bigl(y_1^{-2}\bigr)=-2vy_1^{-2}$. Since \begin{equation*}\mathrm{PPV}(\mathcal{D}/F)=F\langle y_2\rangle_\Delta=F\big\langle\tfrac{y_2}{y_1}\big\rangle_\Delta\end{equation*} and $\delta_x\bigl(\frac{y_2}{y_1}\bigr)=y_1^{-2}$, we apply Proposition~\ref{alg2} with $\eta=y_1^{-2}$ and $\theta=\frac{y_2}{y_1}$ to deduce that $B=B^u$ in this case also.\end{proof}
 
 
 \subsection{Infinite dihedral case}
Suppose that $R_\mathcal{D}(U)=0$ does not have a solution in $K$, but $R_\mathcal{D}(u)=0$ for some $u$ in a quadratic extension of $K$. In this case, one can find $\phi\in\overline{C(t)}(x)$ (see \cite{ulmer-weil:2000, singer:1981}, \cite[Prop. 4.24]{vanderput-singer:2003}) such that the minimal polynomial of $u$ over $K$ is: \vspace{-.125in}\begin{equation*} U^2 +\phi U +\left(\tfrac{1}{2}\delta_x\phi +\tfrac{1}{2}\phi^2-r\right)=0. \end{equation*}
 
 
 \begin{algo}[(Infinite dihedral)]\label{dihedral}$\hphantom{X}$
\begin{center}
\begin{boxedminipage}{10.2cm}\textbf{Input:} $\phi\in \overline{C(t)}(x)$.

\hypertarget{output4}{\textbf{Output:}} A differential algebraic group $\mathbf{D}_\infty(K_0;A^\phi)$, defined by \begin{equation*}\mathbf{D}_\infty(K_0;A^\phi):= \left\{\begin{pmatrix}  a & 0 \\ 0 & a^{-1} \end{pmatrix}, \begin{pmatrix} 0 & a \\ -a^{-1} & 0 \end{pmatrix}\  \middle|\  a\in A^\phi\right\},\end{equation*} and such that $\pgal(\mathcal{D}/K)\simeq\mathbf{D}_\infty(K_0;A^\phi).$\end{boxedminipage}
\end{center}

 \begin{asparaenum}[\bfseries {Step} 1:]

 \item \label{step4.1}Set \vspace{-.12in}\begin{equation*}  \eta:=\tfrac{1}{2}\partial_t\sqrt{4r-2\delta_x\phi-\phi^2}. \end{equation*}
 
 \item \label{step4.2}Apply Algorithm~\ref{exponential} with inputs $p:=\tfrac{\delta_x\eta}{\eta}$ and $q:=\tfrac{\partial_t\eta}{\eta}$, and let $\mathcal{L}^{p,q}$ be the \hyperlink{output2}{Output}. 
 
 \item \label{step4.3}Set $A^\phi:=\mathbb{G}_m(K_0;\mathcal{L}^{p,q})$, and go to \hyperlink{output4}{Output}.
\end{asparaenum}
\end{algo}


 \begin{prop}[(Algorithm~\ref{dihedral} is correct)]\label{dihedral-proof}
 Let $u\in\overline{C(x,t)}$ and $\phi\in\overline{C(t)}(x)$, such that $R_\mathcal{D}(u)=0$ and the minimal polynomial of $u$ over $K$ is given by \vspace{-.06in}\begin{equation*} U^2+\phi U+ \left(\tfrac{1}{2}\delta_x\phi +\tfrac{1}{2}\phi^2-r\right).\vspace{-.06in}\end{equation*} Then, $\pgal(\mathcal{D}/K)\simeq\mathbf{D}_\infty(K_0;A^\phi)$. \end{prop}
 
 \begin{proof}[Proof. (cf. {\cite[\S2.3]{dreyfus:2011}})] If $v$ is the other root of the minimal polynomial of $u$, one can show that there is a basis $\{y_1,y_2\}$ of $\mathrm{Sol}(\mathcal{D})$ such that $\delta_xy_1=uy_1$ and $\delta_xy_2=vy_2$. Since $K(u)\subset\mathrm{PPV}(\mathcal{D}/K)$, Theorem~\ref{pgal} gives a surjection $\pgal(\mathcal{D}/K)\twoheadrightarrow\mathrm{Gal}\bigl(K(u)/K\bigr)$. After possibly replacing $y_2$ by $ay_2$ for some $a\in K_0^\times$, one can show that $\tau(y_1):=y_2$, $\tau(y_2):=-y_1$, $\tau(u):=v$ defines a lift $\tau\in\pgal(\mathcal{D}/K)$ of the unique non-trivial $K$-automorphism of $K(u)$ \cite[\S4.3]{kovacic:1986}. For $\sigma\in\pgal(\mathcal{D}/K)$ we define $\pi(\sigma):=0$ if $\sigma(u)=u$, and $\pi(\sigma):=1$ if $\sigma(u)=v$. If we define \begin{equation*} a_\sigma:=\left(\frac{\sigma\circ\tau^{\pi(\sigma)}(y_1)}{y_1}\right),\end{equation*}a computation shows that $a_\sigma\in K_0^\times$. One can show that the embedding of $\pgal(\mathcal{D}/K)$ in $\mathrm{SL}_2(K_0)$, determined by the choice of basis $\{y_1,y_2\}$ of $\mathrm{Sol}(\mathcal{D})$, identifies $\pgal(\mathcal{D}/K)$ with a subgroup of $\mathbf{D}_\infty(K_0)$, as follows: \begin{equation*} \sigma\mapsto\begin{pmatrix} a_\sigma & 0 \\ 0 & a_\sigma^{-1}  \end{pmatrix}   \begin{pmatrix} 0 & 1 \\ -1 & 0  \end{pmatrix}^{\pi(\sigma)}. \end{equation*}
 
The group $\pgal(\mathcal{D}/K)$ has two connected components, and the normal subgroup $\pgal\bigl(\mathcal{D}/K(u)\bigr)$ is the connected component of the identity. One can show that $\sigma(y_1y_2)=\pm y_1y_2$ for every $\sigma\in\pgal(\mathcal{D}/K)$. Therefore, \begin{equation*}\mathrm{PPV}\smash{\bigl(}\mathcal{D}/K(u)\smash{\bigr)}=K(u)\langle y_1\rangle_\Delta,\end{equation*} and $\pgal\bigl(\mathcal{D}/K(u)\bigr)$ is the $\mathrm{PPV}$~group for the operator $\delta_x-u$ over $K(u)$. For $\sigma\in\pgal\smash{\bigl(}\mathcal{D}/K(u)\smash{\bigr)}$, we have that $a_\sigma=\frac{\sigma y_1}{y_1}$. If we let $A\subset\mathbb{G}_m(K_0)$ be the (isomorphic) image of $\pgal\smash{\bigl(}\mathcal{D}/K(u)\smash{\bigr)}$ under the differential-algebraic homomorphism $\sigma\mapsto a_\sigma$, we have that $\pgal(\mathcal{D}/K)\simeq\mathbf{D}_\infty(K_0;A)$ (defined as in the \hyperlink{output4}{Output} of Algorithm~\ref{dihedral}). We will show that $A=A^\phi$.
 
 As in the upper triangular case, a calculation shows that $\sigma\big(\frac{\partial_ty_1}{y_1}\big)=\frac{\partial_ty_1}{y_1}+\frac{\partial_ta_\sigma}{a_\sigma}$ and $\delta_x\big(\frac{\partial_t y_1}{y_1}\big)=\partial_tu$. Note that we cannot apply Algorithm~\ref{exponential} to compute the $\mathrm{PPV}$~group corresponding to the inhomogeneous equation $\delta_xY=\partial_tu$, because the conditions $\smash{\tfrac{\partial_t^2u}{\partial_tu}}\in K$ and $\smash{\tfrac{\delta_x\partial_tu}{\partial_tu}}\in K$ are not satisfied. However, one can show that $\phi=\sum_i \smash{\frac{c_i}{x-e_i}}$, for some $e_i\in\overline{C(t)}$ and $c_i\in\smash{\tfrac{1}{2}}\mathbb{Z}$ (see \cite[Step 3, \S4.1]{kovacic:1986}), and a computation shows that $\omega:=\sum_i\frac{c_i\partial_te_i}{x-e_i}\in K$ satisfies $\delta_x\omega=\partial_t\phi$. By making the substitution $g:=f+\smash{\tfrac{1}{2}}\omega$, we see that, for any $\mathcal{L}\in K_0[\partial_t]$, \begin{equation*} \text{there exists} \ f\in K(u) \ \text{such that} \ \mathcal{L}(\partial_tu)=\delta_x f \quad\Longleftrightarrow \quad \text{there exists} \ g\in K(u)\ \text{such that} \ \mathcal{L}\left({}\smash[t]{\partial_tu+\tfrac{1}{2}\partial_t\phi}\right)=\delta_xg.\end{equation*} Now we may apply Proposition~\ref{alg2} with $\smash{\eta=\partial_tu+\tfrac{1}{2}\partial_t\phi}$ and $\smash{\theta=\frac{\partial_ty_1}{y_1}+\tfrac{1}{2}\omega}$, because $\smash{\tfrac{\delta_x\eta}{\eta}}$ and $\smash{\tfrac{\partial_t\eta}{\eta}}$ are in $\smash{\overline{C(t)}(x)}$. This concludes the proof that $A=A^\phi$. \end{proof}


\subsection{Recovering the original group} \label{recover}
 
At the beginning of this section, we performed a change of variables on \eqref{original} to put it in the form~\eqref{eq}. We will now indicate how to compute the $\mathrm{PPV}$ group corresponding to \eqref{original}. The corresponding problem in classical Picard-Vessiot theory could be regarded as an exercise; indeed, Kovacic \cite[\S1.1]{kovacic:1986}, as well as numerous other authors, consider the problem solved after computing the group corresponding to~\eqref{eq} only. In the parameterized situation there is a new subtlety, stemming from the fact that $\mathbb{G}_m(K_0)$ has many infinite differential algebraic subgroups (cf. Theorem~\ref{dag2}); while any proper algebraic subgroup of $\mathbb{G}_m(K_0)$ must be finite cyclic. In general, the differential-algebraic relations amongst the solutions for \eqref{eq} and the solution for the change-of-variables operator (i.e., the operator $\mathcal{E}$ defined below), form an obstruction to expressing the $\mathrm{PPV}$~group of \eqref{original} as an almost direct product (that is, a finite-index quotient of the direct product) of the $\mathrm{PPV}$~groups corresponding to these two operators. This is drastically different from the classical situation, where the $\mathrm{PV}$-group of~\eqref{original} is always an almost-direct product of the $\mathrm{PV}$-group of~\eqref{eq} and the change-of-variables group.

To fix notation, consider the differential operators \begin{equation*}\mathcal{D}:=\delta_x^2-r,\qquad \mathcal{E}:=\delta_x+\tfrac{1}{2}r_1,\qquad \text{and}\qquad\mathcal{F}:=\delta_x^2+r_1\delta_x+r_2,\end{equation*} where $r:=\tfrac{1}{4}r_1^2+\tfrac{1}{2}\delta_xr_1-r_2$. Let $E:=\mathrm{PPV}\bigl(\mathcal{E}/\mathrm{PPV}(\mathcal{D}/K)\bigr)$, and let $w\in E^\times$ such that $\mathcal{E}w=0$. In other words, $K\langle w\rangle_\Delta= \mathrm{PPV}(\mathcal{E}/K)\subset E$. Let $y_1,y_2\in\mathrm{PPV}(\mathcal{D}/K)$ be a $K_0$-basis of $\mathrm{Sol}(\mathcal{D})$. If we let $z_i:=wy_i\in E$, one can show that $\mathcal{F}z_i=0$, and since $z_1$ and $z_2$ are $K_0$-linearly independent, they form a basis of $\mathrm{Sol}(\mathcal{F})$ in $E$. Therefore, \begin{equation*} K\langle z_1,z_2\rangle_\Delta=\mathrm{PPV}(\mathcal{F}/K)\subset E=K\langle y_1,y_2,w\rangle_\Delta. \end{equation*} Moreover, a computation shows that $z_1\delta_xz_2-z_2\delta_xz_1=aw^2$ for some $a\in K_0^\times$ \cite[Ex. 1.14 (5)]{vanderput-singer:2003}, so either $E=\pgal(\mathcal{F}/K)$, or else $E$ is a quadratic extension of $\mathrm{PPV}(\mathcal{F}/K)$, generated by $w$. For later use, we define $\nu:=[E:\mathrm{PPV}(\mathcal{F}/K)]$, the algebraic degree of this field extension. One can show that $\nu=1$ if and only if $\pgal(\mathcal{E}/K)$ is finite of odd order. Now consider the $\Delta$-subfield of $E$ defined by $L:=\mathrm{PPV}(\mathcal{D}/K)\cap\mathrm{PPV}(\mathcal{E}/K)$. Since $\pgal(\mathcal{E}/K)$ is a differential algebraic subgroup of $\mathbb{G}_m(K_0)$, it is abelian, and the Kolchin-closed subgroup corresponding to $L$, as a $\Delta$-subfield of  $\mathrm{PPV}(\mathcal{E}/K)$, is normal. Therefore, $L$ is a $\mathrm{PPV}$ extension of $K$ (for some operator). If we let $\Lambda:=\pgal(L/K)$, we have a lattice of $\mathrm{PPV}$ extensions over $K$:

\vspace{-.2in}\begin{equation}  \label{lattice}\vspace{-.1in}
\qquad\xymatrix{ 
& E \ar@{-}[dl] \ar@{-}[dr] \ar@{-}[dd] \ar@/^/@{-}[drrr]^{\mu_\nu}\\
 \mathrm{PPV}(\mathcal{D}/K) \ar@{-}[dr] \ar@{-}[ddr]_(.4){\pgal(\mathcal{D}/K)} & & \mathrm{PPV}(\mathcal{E}/K) \ar@{-}[dl]  \ar@{-}[ddl]^(.4){\pgal(\mathcal{E}/K)} & & \mathrm{PPV}(\mathcal{F}/K)\ar@/^/@{-}[ddlll]^(.3136){\hphantom{XX}\pgal(\mathcal{F}/K)} \\
& L \ar@{-}[d]_(.38){\Lambda} \\ 
 & K
}
\end{equation}\vspace{.05in}

The Galois correspondence (Theorem~\ref{pgal}) implies that \begin{equation*} \pgal(E/K)\simeq\pgal(\mathcal{D}/K)\times_{\Lambda}\pgal(\mathcal{E}/K), \end{equation*}where the fibered product is taken with respect to the surjections $\varphi:\pgal(\mathcal{D}/K) \twoheadrightarrow \Lambda$ and $\psi:\pgal(\mathcal{E}/K)\twoheadrightarrow \Lambda$, induced from the lattice~\eqref{lattice}. It only remains to compute the middle term in the exact sequence \begin{equation}\label{gal-exact} 1\longrightarrow\mu_\nu\longrightarrow\pgal(\mathcal{D}/K)\times_{\Lambda}\pgal(\mathcal{E}/K)\longrightarrow\pgal(\mathcal{F}/K)\longrightarrow 1. \end{equation}

The computation of $\pgal(\mathcal{E}/K)$ is analogous to the computation of the diagonal part of $\pgal(\mathcal{D}/K)$ in the upper triangular case (cf. the proof of Proposition~\ref{borel-proof}): we first compute the $\mathrm{PV}$ group $\gal(\mathcal{E}/K)$. If it is finite, then $\pgal(\mathcal{E}/K)=\gal(\mathcal{E}/K)$. If $\gal(\mathcal{E}/K)$ is infinite, one shows that $\sigma\mapsto\frac{\sigma w}{w}=:e_\sigma$ induces an isomorphism of $\pgal(\mathcal{E}/K)$ onto a differential algebraic subgroup of $\mathbb{G}_m(K_0)$, and we have that $\smash{\sigma\bigl(\frac{\partial_t w}{w}\bigr)=\frac{\partial_tw}{w}+\frac{\partial_te_\sigma}{e_\sigma}}$ and $\delta_x\bigl(\frac{\partial_tw}{w}\bigr)=-\tfrac{1}{2}\partial_tr_1$. An application of Algorithm~\ref{rational}, with input $\eta:=-\tfrac{1}{2}\partial_tr_1$, will give $\mathcal{L}^\eta\in C(t)[\partial_t]$ such that $\pgal(\mathcal{E}/K)\simeq\mathbb{G}_m(K_0;\mathcal{L}^\eta)$. We will now compute $\Lambda$, as well as the maps $\varphi$ and $\psi$. There are two cases to consider, depending on whether $\pgal(\mathcal{E}/K)$ is finite or infinite. We will apply the following result in the case that $\pgal(\mathcal{E}/K)$ is finite.

\begin{lem} The groups $\gal(\mathcal{D}/K)$ and $\pgal(\mathcal{D}/K)$ have the same finite quotients. \end{lem}

\begin{proof}If $G$ is a linear differential algebraic group, we denote the connected component of the identity of $G$ by $G^0$. If $\varphi:G\twoheadrightarrow H$ is a $\partial_t$-algebraic surjection onto a finite group $H$, then $\varphi$ factors through $G\twoheadrightarrow G/G^0$ \cite[p. 906]{cassidy:1972}. Theorem~\ref{zariski} implies that $\pgal(\mathcal{D}/K)/\pgal(\mathcal{D}/K)^0\simeq\gal(\mathcal{D}/K)/\gal(\mathcal{D}/K)^0$.\end{proof}

\emph{Case 1:} Suppose that $\gal(\mathcal{E}/K)=\mu_m$ for some $m\in\mathbb{N}$, so that $\Lambda$ is also finite, the order of $\Lambda$ divides $m$, and $w^m\in K$ \cite{kolchin:1976}. We shall list all the finite cyclic quotients of the algebraic subgroups of $\mathrm{SL}_2(K_0)$, in increasing order of the largest finite cyclic group which they admit as a quotient, and examine the possibilities for $\Lambda$ in each case.

The groups $\mathbb{G}_a(K_0)$, $\mathbb{G}_m(K_0)$, $\mathrm{UT}(K_0)$, and $\mathrm{SL}_2(K_0)$ are connected, and the alternating group $A_5$ is simple, so none of these groups admit a nontrivial finite cyclic quotient. Therefore, if $\gal(\mathcal{D}/K)$ is isomorphic to any of these groups, then $\Lambda$ must be trivial, and we have that \begin{equation*}\pgal(\mathcal{F}/K)\simeq\mu_{m/\nu}\cdot\pgal(\mathcal{D}/K)\subset\mathrm{GL}_2(K_0),\end{equation*} where  $\mu_{m/\nu}\cdot\pgal(\mathcal{D}/K):=\{\zeta\cdot\sigma \ | \ \zeta\in\mu_{m/\nu},\sigma\in\pgal(\mathcal{D}/K)\}$ and $\mu_{m/\nu}\subset C^\times$ denotes the $(m/\nu)^{\text{th}}$ roots of unity.

Each of the following groups admits precisely one nontrivial finite cyclic quotient, of order~$2$: the dihedral groups $\mathbf{D}_{2n}$, the infinite dihedral group $\mathbf{D}_\infty(K_0)$, and the symmetric group $S_4$. Therefore, both $\varphi$ and $\psi$ are determined by $\Lambda$ in any of these cases. If $m$ is odd, we conclude that $\Lambda$ is trivial, and $\pgal(\mathcal{F}/K)\simeq\mu_m\cdot\pgal(\mathcal{D}/K)$. If $m$ is even, one checks whether the (unique) quadratic extension of $K$ which is contained in $\mathrm{PV}(\mathcal{D}/K)$ is (algebraically) isomorphic to $K(w^{\frac{m}{2}})$ using finite Galois theory. This holds if and only if $\Lambda=\mu_2$.

Similarly, the group $A_4$ has a unique cyclic quotient, of order $3$. If $m$ is not divisible by $3$, then $\Lambda=\{1\}$ and we are done. Otherwise, one needs to decide whether the unique cubic extension of $K$ which is contained in $\mathrm{PV}(\mathcal{D}/K)$ is isomorphic to $K(w^{\frac{m}{3}})$; this will hold precisely when $\Lambda=\mu_3$.

Lastly, for $n\in\mathbb{N}$, the groups of the form\begin{equation*} \left\{\begin{pmatrix}  a & b \\ 0 & a^{-1}  \end{pmatrix}\ \middle|\  a\in \mu_n, b\in B\right\}, \end{equation*} whether $B=\{0\}$ or $B=\mathbb{G}_a(K_0)$, each have precisely one quotient of order $l$, for each $l$ dividing $n$, given by $\left(\begin{smallmatrix} a & b \\ 0 & \smash{a^{-1}}\vphantom{a^{-1}} \end{smallmatrix}\right)\mapsto  a^{\frac{n\varepsilon}{l}}$ for some $\varepsilon\in \mathbb{N}$ relatively prime to $\frac{n}{l}$. To compute $\Lambda$, we proceed as before: one has to find the largest $l\in \mathbb{N}$ such that: $(i)$ $l$ divides $\mathrm{gcd}(m,n)$; and $(ii)$ the unique finite cyclic extensions of $K$ of degree $l$, which are contained respectively in $\mathrm{PV}(\mathcal{E}/K)$ and $\mathrm{PV}(\mathcal{D}/K)$, are algebraically isomorphic. In this case, we have that $\Lambda=\mu_l$. If there is no such $l\in\mathbb{N}$, then $\Lambda=\{1\}$.

\emph{Case 2:} Now suppose that $\pgal(\mathcal{E}/K)$ is infinite, and therefore Kolchin-connected. In this case, $\Lambda$ must either be trivial or infinite. The only differential algebraic subgroups of $\mathrm{SL}_2(K_0)$ which admit an infinite abelian quotient are the groups $\mathrm{UT}(K_0;A,B)$ (defined as in the \hyperlink{output3}{Output} of Algorithm~\ref{borel}). Therefore, $\Lambda=\{1\}$ whenever $\pgal(\mathcal{D}/K)$ is not of this form. The following algorithm computes $\Lambda$ and the maps $\varphi$ and $\psi$. For simplicity, Algorithm~\ref{recover-gp} assumes that $\pgal(\mathcal{D}/K)$ is upper-triangular, and that $\pgal(\mathcal{E}/K)$ is infinite and has already been computed.


\begin{algo}[(Recovering the original group in the upper triangular case)]\label{recover-gp} $\hphantom{X}$

\begin{center}
\begin{boxedminipage}{14.8cm}\textbf{Inputs:} $\mathcal{D}=\delta_x^2-r$ and $\mathcal{E}=\delta_x+\frac{1}{2}r_1$, with $r,r_1\in C(x,t)$.

\hypertarget{output6}{\textbf{Output:}}A linear differential algebraic group $\Lambda^{\mathcal{D},\mathcal{E}}$, and differential-algebraic homomorphisms \begin{equation*}\varphi^{\mathcal{D},\mathcal{E}}:\pgal(\mathcal{D}/K)\twoheadrightarrow\Lambda^{\mathcal{D},\mathcal{E}}\qquad\text{and}\qquad \psi^{\mathcal{D},\mathcal{E}}:\pgal(\mathcal{E}/K)\twoheadrightarrow\Lambda^{\mathcal{D},\mathcal{E}}\end{equation*} such that $\Lambda^{\mathcal{D},\mathcal{E}}$ is the $\mathrm{PPV}$ group of $\mathrm{PPV}(\mathcal{D}/K)\cap\mathrm{PPV}(\mathcal{E}/K)$, and $\varphi^{\mathcal{D},\mathcal{E}}$ and $\psi^{\mathcal{D},\mathcal{E}}$ are the corresponding diferential-algebraic surjections of $\mathrm{PPV}$ groups.\end{boxedminipage}
\end{center}

\begin{asparaenum}[\bfseries {Step} 1:]

\item \label{step6.1} Compute a solution $u\in\overline{C(t)}(x)$ to the Riccati equation $R_\mathcal{D}(U)=0$. Apply Algorithm~\ref{borel} with input $u$, and let $\pgal(\mathcal{D}/K)\simeq\mathrm{UT}(K_0;A^u,B^u)$ be the \hyperlink{output3}{Output}. 

\item \label{step6.2} If $A^u\subseteq\mu_2$ , compute a solution $y_1\in\overline{C(x,t)}$ to the differential equation $\delta_xY=uY$. Apply Algorithm~\ref{comparison} with inputs $\eta_1:=y_1^{-2}$, $\eta_2:=-\frac{1}{2}\partial_tr_1$, and let $\Lambda^{\eta_1,\eta_2}$, $\pi_1^{\eta_1,\eta_2}$, and $\pi_2^{\eta_1,\eta_2}$ be the \hyperlink{output5}{Outputs}. Set \begin{equation*}\Lambda^{\mathcal{D},\mathcal{E}}  :=\Lambda^{\eta_1,\eta_2}, \qquad \varphi^{\mathcal{D},\mathcal{E}}  :=\pi_1^{\eta_1,\eta_2}, \qquad\text{and}\qquad \psi^{\mathcal{D},\mathcal{E}}  :\sigma\mapsto\pi_2^{\eta_1,\eta_2}\smash{\left(\frac{\partial_t\sigma}{\sigma}\right)},\vspace{-.06in}\end{equation*} and then go to \hyperlink{output6}{Output}.

If $A^u\not\subset\mu_2$, proceed to Step~\ref{step6.3}.

\item \label{step6.3} If $A^u=\mu_n$ for some $n\geq 3$, set $\Lambda^{\mathcal{D},\mathcal{E}}:=\{1\}$, let $\varphi^{\mathcal{D},\mathcal{E}}$, $\psi^{\mathcal{D},\mathcal{E}}$ be the trivial maps, and go to \hyperlink{output6}{Output}.

If $A^u$ is infinite, proceed to Step~\ref{step6.4}.

\item \label{step6.4} If there exist integers\footnote{See \cite[Prop. 2.4]{compoint-singer:1999} for an algorithm to compute these integers, if they exist.} $m_1, m_2\in\mathbb{Z}$, not both zero, and $f\in\overline{C(t)}(x)$, such that \begin{equation}\label{silly-comp}m_1u-\frac{m_2}{2}r_1=\frac{\delta_xf}{f},\end{equation} there is a unique pair of integers $m_1$ and $m_2$ satisfying~\eqref{silly-comp} with $m_1$ positive and as small as possible. Set \begin{equation*}\Lambda^{\mathcal{D},\mathcal{E}} :=\{\sigma^{m_1} \ | \ \sigma\in A^u\}=\{\sigma^{m_2} \ | \ \sigma\in\pgal(\mathcal{E}/K)\}, \qquad \varphi^{\mathcal{D},\mathcal{E}} :\sigma\mapsto\sigma^{m_1}, \qquad\text{and}\qquad \psi^{\mathcal{D},\mathcal{E}} :\sigma\mapsto\sigma^{m_2},\end{equation*} and go to \hyperlink{output6}{Output}.

Otherwise, proceed to Step~\ref{step6.5}.

\item \label{step6.5} Apply Algorithm~\ref{comparison} with inputs $\eta_1:=\partial_tu$, $\eta_2:=-\frac{1}{2}\partial_tr_1$, and let $\Lambda^{\eta_1,\eta_2}$, $\pi_1^{\eta_1,\eta_2}$, and $\pi_2^{\eta_1,\eta_2}$ be the \hyperlink{output5}{Outputs}. Set \begin{equation*}\smash{\Lambda^{\mathcal{D},\mathcal{E}} :=\Lambda^{\eta_1,\eta_2}, \qquad \varphi^{\mathcal{D},\mathcal{E}}  :\sigma\mapsto\pi_1^{\eta_1,\eta_2}\left(\frac{\partial_t\sigma}{\sigma}\right),\qquad\text{and}\qquad \psi^{\mathcal{D},\mathcal{E}} :\sigma\mapsto\pi_2^{\eta_1,\eta_2}\left(\frac{\partial_t\sigma}{\sigma}\right),}\end{equation*} and then go to \hyperlink{output6}{Output}.

\end{asparaenum}
\end{algo}


Our proof that Algorithm~\ref{recover-gp} gives the right answer will rely on the following result.

\begin{lem}\label{extensions}
Let $\mathcal{E}:=\delta_x-g$ for some $g\in K$, and let $0\neq\rho\in\mathrm{PPV}(\mathcal{E}/K)$ be a solution (i.e., $\delta_x\rho=g\rho$). Let $\ell$ be the order of the differential operator $\mathcal{L}\in K_0[\partial_t]$ such that $ \pgal(\mathcal{E}/K)\simeq\mathbb{G}_m(K_0;\mathcal{L})$. Then, $\mathrm{PPV}(\mathcal{E}/K)$ has transcendence degree $\ell+1$ over $K$, and it is \emph{algebraically} generated by $\rho$ and $\tfrac{\partial_t\rho}{\rho},\partial_t\bigl(\tfrac{\partial_t\rho}{\rho}\bigr),\dots , \partial_t^{\ell-1}\bigl(\tfrac{\partial_t\rho}{\rho}\bigr)$. \end{lem}

\begin{proof} Since $\delta_x\rho=g\rho $, we have that in fact $K\langle \rho\rangle_\Delta=\mathrm{PPV}(\mathcal{E}/K)$ is generated over $K$ by $\{\partial_t^n\rho\}$, for $n\in\mathbb{Z}_{\geq 0}$. That is, $\mathrm{PPV}(\mathcal{E}/K)=K\langle \rho\rangle_{\partial_t}$. A computation shows that \begin{equation*} \partial_t^{n-1}\big(\tfrac{\partial_t\rho}{\rho}\big)-\tfrac{\partial_t^n\rho}{\rho}\in K\bigl(\rho,\partial_t\rho,\dots,\partial_t^{n-1}\rho\bigr), \end{equation*} and therefore $K\langle \rho\rangle_{\partial_t}=K(\rho)\smash{\bigl\langle\tfrac{\partial_t\rho}{\rho}\bigr\rangle_{\partial_t}}$. One can show that $\smash{\delta_x\partial_t^n\bigl(\frac{\partial_t\rho}{\rho}\bigr)=\partial_t^{n+1}g\in K}$ for every $n\geq 0$, and our result now follows from Theorem~\ref{kolostro}, since $\mathcal{L}$ is the operator of smallest order such that $\smash{\mathcal{L}\bigl(\tfrac{\partial_t\rho}{\rho}\bigr)\in K}$ (cf. Lemma~\ref{singer1} and the proof of Proposition~\ref{borel-proof}).\end{proof}

\begin{prop}[(Algorithm~\ref*{recover-gp} is correct)] Suppose that $\pgal(\mathcal{D}/K)$ is isomorphic to $\mathrm{UT}(K_0;A,B)$, and that $\pgal(\mathcal{E}/K)$ is infinite. Define the fields $L$ and $E$ as above, that is \begin{equation*} E  := \mathrm{PPV}\bigl(\mathcal{E}/\mathrm{PPV}(\mathcal{D}/K)\bigr), \qquad \text{and} \qquad L  := \mathrm{PPV}(\mathcal{D}/K)\cap\mathrm{PPV}(\mathcal{E}/K)\subseteq E. \end{equation*}Then, the \hyperlink{output6}{Output} of Algorithm~\ref{recover-gp} is correct. That is, $\Lambda\simeq\Lambda^{\mathcal{D},\mathcal{E}}$, $\varphi=\varphi^{\mathcal{D},\mathcal{E}}$, and $\psi=\psi^{\mathcal{D},\mathcal{E}}$.
\end{prop}

\begin{proof} Let $y_1,y_2\in \mathrm{PPV}(\mathcal{D}/K)$ be the solutions computed in Proposition~\ref{borel-proof}, that is, $y_1$ and $y_2$ satisfy the relations $\delta_xy_1=uy_1$ and $\delta_x\bigl(\frac{y_2}{y_1}\bigr)=y_1^{-2}$. We will divide the proof in two cases, depending on whether $A$ is finite or infinite.

\emph{Case 1:} If $A =\mu_n$ for some $n\in\mathbb{N}$, the Galois correspondence implies that $y_1^n\in K$, and we showed in the proof of Proposition~\ref{borel-proof} that $\mathrm{PPV}(\mathcal{D}/K)=K(y_1)\smash{\bigl\langle \tfrac{y_2}{y_1}\bigr\rangle_\Delta}$. Suppose that $B =\mathbb{G}_a(K_0;\mathcal{L}_1)$, where $0\neq\mathcal{L}_1\in\overline{C(t)}[\partial_t]$, and let $\ell_1:=\mathrm{ord}(\mathcal{L}_1)$. If we let $\theta_1:=\frac{y_2}{y_1}$, the argument given at the beginning of the proof of Proposition~\ref{comparison-proof} shows that in fact\vspace{-.12in} \begin{equation*}\mathrm{PPV}(\mathcal{D}/K)=K(y_1,\theta_1,\partial_t\theta_1,\dots,\partial_t^{\ell_1-1}\theta_1).\end{equation*}  On the other hand, if we let $\mathcal{L}_2\in C(t)[\partial_t]$ such that $\pgal(\mathcal{E}/K)\simeq\mathbb{G}_m(K_0;\mathcal{L}_2)$, and $\theta_2:=\frac{\partial_tw}{w}$, we may apply Lemma~\ref{extensions} to $\mathrm{PPV}(\mathcal{E}/K)$ to conclude that \begin{equation*} \mathrm{PPV}(\mathcal{E}/K)=K( w,\theta_2,\partial_t\theta_2, \dots,\partial_t^{\ell_2-1}\theta_2), \end{equation*} where $\ell_2:=\mathrm{ord}(\mathcal{L}_2)$. Theorem~\ref{kolostro} shows that \begin{equation*}K(y_1)\cap K(\theta_ s ,\partial_t\theta_ s ,\dots,\partial_t^{\ell_s -1}\theta_s )  =K= K(w)\cap K(\theta_s ,\partial_t\theta_s ,\dots,\partial_t^{\ell_s -1}\theta_s ), \end{equation*}for each $ s =1,2$. Moreover, since $y_1$ is algebraic and $w$ is transcendental (over $K$), another application of Theorem~\ref{kolostro} gives that $K(y_1)\cap K(w)=K$. Therefore, if we let $p_1:=\tfrac{\delta_x(y_1^{-2})}{y_1^{-2}}$, $p_2:=-\tfrac{\delta_x\partial_tr_1}{\partial_tr_1}$, and $\mathcal{P}_s :=\delta_x^2-p_s \delta_x$ for $ s =1,2$ (we remark that these are the homogeneous operators corresponding to the first order inhomogeneous equations $\delta_xY=y_1^{-2}$ and $\delta_xY=-\frac{1}{2}\partial_tr_1$, respectively), we have that\begin{equation*} \mathrm{PPV}(\mathcal{P}_s /K)=K\langle\theta_s \rangle_{\partial_t} \subset E,\qquad \text{and}\qquad L=\mathrm{PPV}(\mathcal{P}_1/K)\cap\mathrm{PPV}(\mathcal{P}_2/K)\subset E.\end{equation*} If $n\leq 2$, then $y_1^{-2}\in K$, in which case Proposition~\ref{comparison-proof} establishes our contention. We claim that if $n\geq 3$, then $L=K$. To see this, we proceed by contradiction: suppose that $L\neq K$. By Theorem~\ref{kolostro}, the elements $\theta_1,\theta_2,\partial_t\theta_1,\partial_t\theta_2,\dots,\partial_t^{\ell_1-1}\theta_1,\partial_t^{\ell_2-1}\theta_2\in E$ must satisfy a relation of the form \begin{equation}\label{recover-proof-1} \sum_{i=0}^{\ell_1-1}a_i\partial_t^i\theta_1 -\sum_{j=0}^{\ell_2-1}b_j\partial_t^j\theta_2=f\in K, \end{equation} with $a_0,\dots,a_{\ell_1-1},b_0,\dots,b_{\ell_2-1}\in K_0$ not all zero. Applying $\delta_x$ on both sides of the equality in~\eqref{recover-proof-1} gives a contradiction, since $\sum_ia_i\partial_t^i\eta_1=(\sum_ia_iR_i)y_1^{-2}\notin K$, where the sequence $\{R_i\}\subset K$ is defined as in Step~\ref{step2.3} of Algorithm~\ref{exponential}.

\emph{Case 2:} If $A$ is infinite, then $y_1$ is transcendental over $K$. Let $\mathcal{L}_3\in \overline{C(t)}[\partial_t]$ such that $A =\mathbb{G}_m(K_0;\mathcal{L}_3)$, and let $\theta_3:=\frac{\partial_ty_1}{y_1}$. If we apply Lemma~\ref{extensions} to the field $K\langle y_1\rangle_\Delta=~\mathrm{PPV}\smash{\bigl((\delta_x-u)/K\bigr)}$ with $\rho=y_1$, we obtain\begin{equation*} \smash{K\langle y_1\rangle_\Delta=K(y_1,\theta_3,\partial_t\theta_3,\dots,\partial_t^{\ell_3-1}\theta_3)},\end{equation*} where $\ell_3:=\mathrm{ord}(\mathcal{L}_3)$. We claim that $L=K\langle y_1\rangle_\Delta\cap K\langle w\rangle_\Delta\subset E$ (in other words, we may ignore the second solution $y_2\in\mathrm{PPV}(\mathcal{D}/K)$ in this case). To see this, let $\theta_1:=\smash{\frac{y_2}{y_1}}$ and $\theta_2:=\smash{\frac{\partial_tw}{w}}$ as above. Since every element of $(K\langle y_1,y_2\rangle_\Delta\cap K\langle w\rangle_\Delta)\backslash K\langle y_1\rangle_\Delta$ defines a nontrivial algebraic relation over $K\langle y_1\rangle_\Delta$ amongst the elements $\theta_1,\theta_2,\partial_t\theta_1,\partial_t\theta_2,\dots,\partial_t^{\ell_1-1}\theta_1,\partial_t^{\ell_2-1}\theta_2$ (cf. the proof of Proposition~\ref{comparison-proof}), and since $\delta_x(\partial_t^i\theta_1),\delta_x(\partial_t^j\theta_2)\in K\langle y_1\rangle_\Delta$, Theorem~\ref{kolostro} implies that, if $K\langle w\rangle_\Delta\cap\ K\langle y_1,y_2\rangle_\Delta\not\subset K\langle y_1\rangle_\Delta$, then there is a nontrivial linear combination \begin{equation}\label{bad4}\sum_{i=0}^{\ell_1-1}a_i\partial_t^i\theta_1-\sum_{j=0}^{\ell_2-1}b_j\partial^j\theta_2=f\in K\langle y_1\rangle_\Delta\end{equation} with $a_i,b_j\in K_0$. Define the sequence $\{R_i\}\subset K$ by $R_0:=1$, $R_i:=\partial_tR_{i-1}-2\theta_3R_{i-1}$ for $i\geq 1$ as in Step~\ref{step2.3} of Algorithm~\ref{exponential}. Applying $\delta_x$ on both sides of the equality in \eqref{bad4}, we obtain \begin{equation}\label{bad5} \left(\sum_{i=0}^{\ell_1-1}a_iR_i\right)y_1^{-2} +\frac{1}{2} \sum_{j=0}^{\ell_2-1}b_j \partial_t^{j+1}r_1=\delta_xf. \end{equation} There is a unique partial fraction decomposition: $f=\sum_kc_ky_1^k+\sum_{m,l} \frac{d_{m,l}}{(y_1-e_m)^l}$ with $c_k\in K(\theta_3,\dots,\partial_t^{\ell_3-1}\theta_3)$, and $d_{l,m}$, $e_l$ algebraic over this field. Therefore, if we let $e_0:=0$, \eqref{bad5} implies that\begin{align*} \delta_xf & = \sum_k(\delta_xc_k+ku)y_1^k+\sum_{m,l}\frac{\delta_xd_{m,l}(y_1-e_m)-ld_{m,l}(uy_1-\delta_xe_m)}{(y_1-e_m)^{l+1}} \\ &= \sum_k(\delta_xc_k+ku)y_1^k+\sum_{m,l}\frac{\delta_xd_{m,l}-lud_{m,l}}{(y_1-e_m)^l}+\frac{ld_{m,l}(\delta_xe_m-ue_m)}{(y_1-e_m)^{l+1}} \\ &= \delta_xc_0 + (\delta_xd_{0,2}-2ud_{0,2})y_1^{-2}.\end{align*}This implies that $\delta_x(\sum_jb_j\partial_t^j\theta_2+c_0)=0$, whence $\sum_jb_j\partial_t^j\theta_2=g-c_0$ for some $g\in K_0$; or, in other words, $\sum_jb_j\partial_t^j\theta_2\in K\langle y_1\rangle_\Delta$. We have shown that if $\sum_jb_j\partial_t^j\theta_2\in K\langle y_1,y_2\rangle_\Delta$, then $\sum_jb_j\partial_t^j\theta_2\in K\langle y_1\rangle_\Delta$. This establishes our claim that $L= K\langle y_1\rangle_\Delta\cap K\langle w\rangle_\Delta$.

Since $\smash{\frac{\delta_xy_1}{y_1},\frac{\delta_xw}{w}}\in K$, and $\delta_x(\partial_t^i\theta_2),\delta_x(\partial_t^j\theta_3)\in K$, Theorem~\ref{kolostro} implies that \begin{equation*}K(y_1)\cap K(\theta_s ,\partial_t\theta_s ,\dots,\partial_t^{\ell_s -1}\theta_s )=K=K(w)\cap K(\theta_s ,\partial_t\theta_s ,\dots,\partial_t^{\ell_s -1}\theta_s ),\end{equation*} for each $ s =2,3$. Therefore, if $L\neq K$, precisely one of the following statements is true:
\begin{enumerate}

\item \label{end-1} $L= K(y_1)\cap K(w)$.

\item \label{end-2} $L= K(\theta_2,\partial_t\theta_2,\dots,\partial_t^{\ell_2-1}\theta_2)\cap K(\theta_3,\partial_t\theta_3,\dots,\partial_t^{\ell_3-1}\theta_3)$.

\end{enumerate}
By Theorem~\ref{kolostro}, \eqref{end-1} holds if and only if there exist integers $m_1$ and $m_2$, not both zero, such that $y_1^{m_1}w^{m_2}~\in~K$, and this holds if and only if the condition of Step~\ref{step6.4} holds. In case \eqref{end-2}, Proposition~\ref{comparison-proof} establishes our result. \end{proof}


\section{Concluding remarks}\label{conclusion}
\subsection{Consequences} We mention two consequences of the algorithms.

\begin{thm} \label{sitthm} Every unimodular $\mathrm{PPV}$~group of a second order homogeneous linear equation over $K$ is one of the groups in Sit's classification. \end{thm}

\begin{thm} \label{wibmerthm}Let $\mathcal{D}\in\overline{C(t)}(x)[\delta_x]$ be an operator of order $1$ or $2$. Then, $\pgal(\mathcal{D}/K)$ is defined over $\overline{C(t)}$. \end{thm}

\subsection{Improvements}In computer algebra, one usually computes the partial-fraction decomposition of a rational function by computing the square-free factorization\footnote{If $F$ is a field of characteristic zero, and $p\in F[x]$ is a polynomial, we say that $p=\prod_{i=1}^n(p_i)^i$ is a square-free factorization of $p$ if $p_i\in F[x]$ for each $i$, and $\mathrm{gcd}_x(p_i,p_j)=1$ for $i\neq j$. The square-free factorization is unique up to multiplication by constant polynomials, and it can be computed using only the Euclidean algorithm and algebraic differentiation.} of its denominator \cite{geddes:1992}. In the algorithms of \S\ref{order1}, we stopped short of computing a complete square-free factorization of the denominators involved, because the polynomial $d$ computed in each algorithm was sufficient to establish our results; namely, that the algorithms work, their outputs are defined over $\overline{C(t)}$, and that no factorizations of polynomials are necessary (beyond those which can be carried out using the Euclidean algorithm). The use of a single denominator $d$ (and its powers) also allowed us to reduce the number of indices involved in our presentation. We believe that this has improved the clarity of our presentation. In practice, one should reduce the size of the system of linear equations considered in each algorithm by working with square-free factorizations of the denominators. 

Each algorithm in \S\ref{order1} constructs a first order inhomogeneous differential equation with undetermined coefficients (these are the equations~\eqref{nrational},~\eqref{ihomh}, and~\eqref{ihom-comp}), and one then finds a minimal set of values for the coefficients such that the differential equation has a rational solution. One could also solve these differential equations abstractly, using any of the usual methods in computer algebra (e.g., Hermite reduction, Rothstein-Trager, Risch$,\dots$) \cite{geddes:1992}, and then find a minimal set of values for the undetermined coefficients such that the solution is rational. Although any of these approaches would have to be abstractly equivalent to ours, they may well be more efficient in some situations. Therefore, the eventual implementation of our algorithms should first decide, perhaps heuristically, how to set up the systems of linear equations in order to minimize the amount of computation.

\subsection{Future directions}
The differential operators produced by the algorithms presented in \S\ref{order1} arise as solutions to a \emph{creative telescoping} problem, namely: given $\eta\in F$, where $F$ is defined as in the beginning of \S\ref{order1}, find $\mathcal{L}\in\overline{C(t)}[\partial_t]$ and $f\in F$ such that $\mathcal{L}(\eta)=\delta_xf$ (see \cite[\S 1]{chen-kauers-singer:2012} for a more precise and general definition); such an operator $\mathcal{L}$ is called a \emph{telescoper} for $\eta$. In \cite{chen-kauers-singer:2012} and \cite{chen-singer:2012}, the authors propose algorithms which, in particular, solve the problem of finding a telescoper $\mathcal{L}\in C(t)[\partial_t]$ for $\eta\in\overline{C(x,t)}$ arbitrary. Lemma~\ref{singer1} replaces a telescoping problem over $F$ with a ``twisted'' telescoping problem over the smaller field $K$ (see part \eqref{tcond} of Lemma~\ref{singer1}), for the restrictive class of algebraic functions $\eta\in\overline{C(x,t)}$ such that $\eta^n\in \overline{C(t)}(x)$ for some $n\in\mathbb{N}$, and also in the case that $\eta$ is transcendental over $K$, provided that $\smash{\frac{\delta_x\eta}{\eta}}$ and $\smash{\frac{\partial_t\eta}{\eta}}$ belong to $\overline{C(t)}(x)$. It would be interesting to see whether one can use creative telescoping to compute $\mathrm{PPV}$ groups for higher-order systems, or over more general fields than we consider here. It would also be interesting to see whether our approach, which handles a very restrictive class of telescoping problems, can be generalized in order to ``descend'' some telescoping problems to simpler base fields, with Lemma~\ref{singer1} as a prototype.

We mentioned in the introduction a result of \cite{gill-gor-ov:2012}, which says that the $\mathrm{PPV}$ groups we wish to compute are actually defined over $C(t)$, while our algorithms find differential polynomial equations which are defined over $\overline{C(t)}$. This is because we work with the solutions to \ref{original} found by Kovacic's algorithm, which are not (and cannot be) defined over $C(t)$ in general. It would be desirable to have algorithms to compute the defining equations for the $\mathrm{PPV}$ group over the smaller field $C(t)$, but we do not yet know whether it is possible to extend our methods to produce such an algorithm.

It would be desirable to have algorithms in the case of several parametric derivations. We do not yet know whether it is possible to extend our methods to this more general setting in a straightforward way. In \cite{gor-ov:2012}, the authors reduce the number of compatibility conditions that one has to check when working with several parametric derivations. It would be interesting to see if the results of \cite{gor-ov:2012} can be used to extend our methods to the setting of several parametric derivations, by working with one parametric derivation at a time.


\begin{acknowledgements} Professor Alexey Ovchinnikov suggested this problem to me last summer, and has been an invaluable source of advice, knowledge and insight ever since. I am very thankful to him for his patience, kindness and generosity, without which this work would not have been possible. I have also benefited enormously from many conversations about this work with the following people: Phyllis Cassidy, Richard Churchill, Thomas Dreyfus, Guillaume Duval, Henri Gillet, Raymond Hoobler, Bill Keigher, Michael Singer, and William Sit. My deep thanks go to all of them for taking the time to help me improve this work. I presented the results in this paper at the Kolchin Seminar in New York, the Joint Mathematics Meetings in Boston, and the Workshop on Differential Schemes and Differential Cohomology in Banff and Calgary, Canada. I thank the audiences of those talks for their very valuable comments, suggestions, and criticisms. \end{acknowledgements}


 \bibliography{kovbib}{}  \bibliographystyle{spmpsci} \nocite{*}
 
 \end{document}